\documentclass[a4paper,11pt]{article}

\usepackage{setup/mysetup}
\usepackage{setup/mysymbols}
\usepackage{setup/myarticle}
\usepackage{setup/customsetup}

\graphicspath{{figures/}}

\title{Managing delay in tail assignment: from minimum turn time to stochastic routing at Air France}

\date{\today}
\author{Léo Baty, Axel Parmentier}

\begin{document}

\maketitle

\begin{abstract}
    On-time performance is a critical challenge in the airline industry, leading to large operational and customer dissatisfaction costs.
    The tail assignment problem builds the sequences of flights or routes followed by individual airplanes.
    While airlines cannot avoid some sources of delay, choosing routes wisely limits propagation along these.
    This paper addresses the stochastic tail assignment problem at Air France.
    We propose a column generation approach for this problem.
    The key ingredient is the pricing algorithm, which is a stochastic shortest path problem.
    We use dedicated bounds to discard paths in an enumeration algorithm, and introduce new bounds based on a lattice ordering of the set of piecewise linear convex functions to strike a balance between bounds quality and computational cost.
    A diving heuristic enables us to retrieve integer solutions.
    Numerical experiments on real-world Air France instances demonstrate that our algorithms lead to an average 0.28\% optimality gap on instances with up to 600 flight legs in a few hours of computing time. The resulting solutions effectively balance operational costs and delay resilience, outperforming previous approaches based on minimum turn time.

\end{abstract}

\section{Introduction}

Aircraft routing represents one of the most impactful applications of Operations Research in the airline industry, with direct consequences for operational efficiency, fuel consumption, and on-time performance.
In a highly competitive industry, optimizing these routes can yield significant competitive advantages.
This problem, solved daily by airlines worldwide, consists in building sequences of flight legs and maintenances operated by aircraft of various fleets, commonly called aircraft routes.
This so-called \emph{aircraft routing} problem can be solved at various time horizons before operations.
Depending on the considered time horizon and distance to operations, available information varies, and so does the problem constraints and objective.
When the horizon is located a few days before the day of operations, the problem is then called \emph{tail assignment}.
In this context, all available aircraft in the fleet are known along with their locations and scheduled maintenances, and therefore sequences of legs and maintenances need to be assigned to each aircraft.

While aircraft routing is usually considered a feasibility problem, tail assignment is in practice an optimization problem.
Aircraft routes from the tail assignment directly affect operational costs (mainly fuel). Airlines therefore want to maximize the usage of aircraft that consume low amount of fuel, which leads to small turn times.
However, low turn times favor delay propagation along these routes, causing additional costs, potentially disrupting hundreds of passengers travel plans.
These costs extend beyond direct operational expenses to include customer compensation, unscheduled emergency crew costs, and damage to the airline's reputation.
If slacks between flight legs are too tight, one flight being late can quickly snowball into all subsequent flights being late.

Planning time buffers between critical flights is a good way to control and prevent delay propagation, but can increase operational costs and affect the ability to operate all legs with the available fleet.
Ensuring a good balance between operational costs and delay costs is therefore crucial in order to build resilient and efficient aircraft routes.

In practice, airlines such as Air France try to mitigate delay propagation by enforcing fixed time buffers between legs through hard constraints.
They do that by artificially increasing the scheduled turn time values between legs, which indirectly ensures low delay propagation along any possible aircraft routes.
The main benefit of this approach is that it is quite simple to implement in practice: the problem can be modeled as a mixed integer linear program that is solved very efficiently.
However, this procedure has several drawbacks.
First, unnecessarily adding time constraints between legs can lead to suboptimal solutions by preventing lower cost solutions feasibility.
Second, the main issue that hinders the usability of this approach in practice is that it often makes instances infeasible.
Indeed, instances that would have been feasible without these additional constraints can become infeasible when these constraints are added.
Not being able to provide a solution such that all legs are operated is not acceptable in an industrial context, and can prevent the usage of optimization in this context.

Previous approaches to this problem have largely focused on either purely deterministic models that ignore delay uncertainty or on robust optimization techniques that can be overly conservative. While some stochastic models exist, they often struggle to balance computational tractability with realistic delay modeling.
We address these issues by introducing a scalable stochastic optimization approach.

\subsection{Related work}

The airline industry extensively uses Operations Research techniques to solve optimization problems on a daily basis (see \cite{barnhartApplicationsOperationsResearch2003} for a general review).
The aircraft routing problem is one of the most studied in the literature.
Many variants exist, among which the tail assignment problem~\parencite{gronkvistTailAssignmentProblem2005a}, which is the focus of this paper.

Most papers focus on variants of the traditionally deterministic version of the problem, by either considering the aircraft routing as a feasibility problem, or by minimizing operational costs.
For these deterministic variants, several pure operations research techniques are state-of-the-art and very efficient for this task.
Most compact mixed integer program formulations are formulated as a multi-commodity flow based on an acyclic digraph modeling the feasible connections between flight legs, introduced by \cite{haneFleetAssignmentProblem1995} and defined in Subsection~\ref{subsec:event-graph}.
Following works provide similar formulations, and improve the efficiency of the algorithms to solve them, such as \cite{khaledCompactOptimizationModel2018} who propose a compact formulation for the tail assignment.
Another popular approach is to use column generation algorithms to solve a route based formulation, such as \cite{gabteniHybridColumnGeneration2006} who propose a hybrid column generation and constraint programming approach for the tail assignment.
Some works combine aircraft routing with another operational problem, such as crew pairing, as \cite{parmentierAircraftRoutingCrew2020} who consider an integrated column generation formulation for aircraft routing and crew pairing problem at Air France.

While these deterministic variants are well studied, delay resilient versions of the tail assignment problem are less common in the literature.
Some approaches model it as a two-stage stochastic optimization problem with recovery and recourse decisions, such as \cite{khaledMulticriteriaRepairRecovery2018} who propose a multi-objective repair real-time replanning approach for aircraft routing, or \cite{glombStochasticOptimizationApproach2023} who propose a Benders decomposition approach for recovery of predictive maintenance for failure scenarios.
In this paper, we want to build resilient routes by minimizing delay costs, that are directly impacted by the delay propagation.
One of the first stochastic formulations for the aircraft routing problem has been proposed by \cite{lanPlanningRobustAirline2006a} with a mixed-integer program with a simple delay model based on historical delay data to solve the robust maintenance routing problem.
Some other studies use robust optimization techniques to model uncertainty, such as \cite{yanRobustAircraftRouting2018} who propose a robust optimization model for the aircraft routing problem with uncertainty sets on delays.
\cite{marlaRobustOptimizationLessons2018} also use robust optimization techniques to model delays, and propose a simulator to generate delay scenarios.

To the best of our knowledge, \cite{biroliniDayaheadAircraftRouting2023} is the only work that addresses a similar stochastic routing problem with delay considerations.
However, their approach differs from ours in several important aspects.
First, their work primarily focuses on developing a sophisticated delay prediction model, using a queuing model combined with machine learning to predict delays, while our focus is on optimization techniques with given delay scenarios.
Second, they consider a one-day horizon, while we consider a one-week horizon, resulting in significantly larger and more complex optimization problems.
Their smaller problem instances make the underlying optimization problem easier to solve, for which exhaustive enumeration of all solutions is sufficient to find the optimal one.
In contrast, our larger instances require more sophisticated optimization techniques, including both a compact MIP formulation and a column generation approach.
Finally, they do not consider operational costs, while we consider both operational and delay costs simultaneously, providing a trade-off between them with additional operational constraints such as aircraft scheduled maintenances.

\subsection{Contributions}

In this paper, we present a novel approach to solve the stochastic tail assignment problem for optimizing operational and delay costs.
\begin{enumerate}
    \item We introduce a stochastic optimization model for the tail assignment with delay. It differs from the literature in two ways: it focuses on tail assignment instead of aircraft routing, and deals with a scenario based version (sample average approximation), which enables to capture dependencies in the delay distribution.
    \item Our main contribution is a column generation with a diving heuristic for the sample average approximation for the stochastic tail assignment problem. 
    \item The pricing algorithm is new, its originality lies in the computation of bounds for faster path enumeration and filtering. We propose and implemented two open-source Julia libraries to support this algorithm: \texttt{ConstrainedShortestPaths.jl} for a generic implementation of constrained shortest paths algorithms, and \texttt{PiecewiseLinearFunctions.jl} for working with piecewise linear functions and computing bounds efficiently.
    \item Numerical experiments on industrial instances from Air France shows that our column generation strongly outperforms compact MILP approaches, and that the diving heuristic is needed to retrieve good quality integer solutions. Solutions in a 0.3\% optimality gap are retrieved in a few hours on instances with up to 600 legs.
\end{enumerate}

\subsection{Outline}
In Section~\ref{sec:tail_assignment_OR:problem_setting}, we introduce our problem setting.
Section~\ref{sec:delay-model} describes our delay propagation model and the sample average approximation approach used to evaluate delay costs with given scenarios.
Then, Section~\ref{sec:column-generation} presents our optimization approach to solve the stochastic tail assignment problem: a column generation algorithm combined with a diving heuristic.
Numerical experiments on real-world Air France instances are presented in Section~\ref{sec:numerical-experiments}, before we conclude in Section~\ref{sec:conclusion}.

\section{Problem statement}\label{sec:tail_assignment_OR:problem_setting}
We now introduce the \emph{stochastic tail assignment problem}.
The goal is to build cost-minimal aircraft routes for a given schedule, such that all legs and maintenances are operated, and all operational constraints are satisfied.
We focus on a tactical problem setting, solved about one week before operations.
Once routes are built, they cannot be changed.
At operations time, in addition to usual operational costs, we observe flight delays that propagate to downstream flights along built aircraft routes.
These random delays cause additional costs, denoted as delay costs.
The objective is to minimize total operational and delay costs, in expectation.
While delays are unknown in advance, we assume knowledge about a scenario set of delays, and use them to minimize the empirical expectation.

\subsection{Activity schedule and fleet}
We consider a set of flight legs $\legSet$, to operate with a given aircraft fleet $\aircraftSet$.
Each aircraft $\aircraft\in\aircraftSet$ of the fleet, has a set of scheduled mandatory maintenances $\maintenanceSubset{\aircraft}$.
The set of all maintenances is denoted by $\maintenanceSet = \cup_{\aircraft\in\aircraftSet}\maintenanceSubset{\aircraft}$.
We call \emph{activity schedule} the set of all legs and maintenances, denoted by $\activitySchedule = \legSchedule \cup \maintenanceSet$.
Each activity $\activity\in\activitySchedule$ in the schedule has a known and fixed \emph{scheduled departure time} $\legScheduledDep{\activity}$ and \emph{scheduled arrival time} $\legScheduledArr{\activity}$, as well as departure and arrival airports, respectively denoted by $\airportDep{\activity}$ and $\airportArr{\activity}$.
Note that for maintenances, departure and arrival airports are identical: $\forall\maintenance\in\maintenanceSchedule,\, \airportDep{\maintenance} = \airportArr{\maintenance}.$

In addition to mandatory maintenances $\maintenanceSubset{\aircraft}$, each aircraft $\aircraft$ has a known \emph{first activity} to operate, which corresponds to its last known scheduled activity, and serves as a starting location.
This first activity is included in the schedule, and denoted by $\lastActivity{i}$.
Finally, we define a \emph{connection} $\connection = (\activity^-, \activity)$ as two successive activities that can be operated by the same aircraft.
We consider a set of \emph{mandatory connections} $\mandatoryConnections$ to operate.
In practice, this represents two legs that need to be operated successively by the same aircraft, for operational reasons.

\subsection{Connection graph}\label{subsec:event-graph}
A practical way to model an instance of the tail assignment problem is to use a directed acyclic digraph $\digraph = (\vertexSetBar, \arcSet)$, called \emph{connection graph}, in which vertices are activities, and arcs are valid connections.
The vertex set is defined as $\vertexSetBar = \activitySchedule\cup\{\source, \sink\}$, i.e. the activity schedule to which are added source and destination dummy vertices $s$ and $t$, respectively.
There is an arc $\arc = (\activity^-, \activity)$ between two activities $\activity^-$ and $\activity$ if and only if connecting airports are compatible ($\airportArr{\activity^-} = \airportDep{\activity}$), and scheduled time are separated by at least some known constant $\minimumTurnTime{\activity}$, called \emph{minimum turn time}: 
\begin{equation}
    \legScheduledArr{\activity^-} + \minimumTurnTime{\activity} \leq \legScheduledDep{\activity}.
\end{equation}
For every activity $\activity\in\activitySchedule$, there is a $(\source, \activity)$ arc and a $(\activity, \sink)$ arc.
Arc $(\source, \sink)$ is also in $\arcSet$.
By construction, this digraph is acyclic.
The arc set $\arcSet$ can be decomposed into a (non-disjoint) union of arc subsets denoted by $\arcSet^\aircraft$ for each aircraft $\aircraft\in\aircraftSet$.
Indeed, for a given aircraft $\aircraft$, many arcs can be removed from $\arcSet$ because they cannot be part of any feasible route for this aircraft.
Typically, for a specific aircraft $\aircraft$, all arcs that would allow skipping a scheduled maintenance of that aircraft, as well as all arcs connected to maintenances of other aircraft, can be safely removed from the arc subset $\arcSubSet{\aircraft}$.

\subsubsection{Feasible decisions}
Under this definition, a feasible aircraft route is an $\source-\sink$ path in $\digraph$, such that all mandatory maintenances and first activities are operated by their corresponding aircraft, i.e., an $\source-\sink$ path in the sub-graph $\disubgraph{\aircraft} = (\vertexSetBar, \arcSubSet{\aircraft})$.
We denote by $\feasibleRouteSet{\aircraft}$ the set of feasible routes for aircraft $\aircraft$.
A feasible solution of the full tail assignment problem is therefore a 
set of feasible routes (one route for each aircraft $\aircraft\in\aircraftSet$) partitioning all activities $\activitySchedule$.
We denote by $\feasibleDecisions{\instance}$ the set of feasible decisions $\decision$ for instance $\instance$.

\subsubsection{Preprocessing the graph}
The connection graph allows us to model the problem in a compact way, encoding most operational constraints directly through the graph structure.
We perform two main types of preprocessing to enforce additional constraints efficiently.

We precompute arc subsets $\arcSubSet{\aircraft}$ for each aircraft $\aircraft\in\aircraftSet$, which enables efficient enforcement of maintenance and first activity constraints.
For a given aircraft $\aircraft$, we can safely remove from $\arcSet$ all arcs that cannot be part of any feasible route for this aircraft.
Specifically, we remove all arcs that would allow skipping a scheduled maintenance of aircraft $\aircraft$, all arcs connected to maintenances of other aircraft, and arcs incompatible with the first activity constraint for aircraft $\aircraft$.

The connection graph can be further preprocessed to enforce mandatory connections constraints through graph sparsification.
A solution is feasible only if, for each mandatory connection arc $\arc\in\mandatoryConnections\subset\arcSet$, an aircraft route in the solution passes through this arc.
To enforce that arc $\arc = (\activity^-, \activity)$ appears in the solution, we remove from $\arcSet$ all outgoing arcs $(\activity^-, w)$ where $w \neq \activity$.
This forces any route containing activity $\activity^-$ to continue through the mandatory connection $\arc$.
With this preprocessing, the set partition constraint automatically enforces the mandatory connections constraint.

A graph traversal algorithm can identify and remove all such redundant arcs efficiently.
In the remainder of this paper, we assume that the connection graph has been preprocessed in this manner.

\subsection{Operational costs}
While aircraft routing is often a feasibility problem (all aircraft are considered identical), tail assignment is an optimization problem.
Individual aircraft can have different engines, yielding different fuel consumption and operational costs, even if they are of the same aircraft type.
For this reason, we denote  by $\legCost$ the \emph{leg cost} for operating leg $\leg\in\legSet$ with aircraft $\aircraft\in\aircraftSet$.
Maintenance costs do not need to be considered, because maintenances are always operated by the same aircraft in any feasible solution.
Additionally, we associate to each arc $\arc$ a \emph{connection cost} $\connectionCost$ when aircraft $\aircraft$ operates connection $\connection$.

The total \emph{operational cost} of a tail assignment is the sum of all leg and connection costs.
The total operational cost $\routeOperationalCost{\route}$ of a feasible route $\route\in\feasibleRouteSet{\aircraft}$ for aircraft $\aircraft$ is the sum of all leg and connection costs in the route:
\begin{equation}
    \routeOperationalCost{\route} = \sum_{\leg\in\route} \legCost[\leg] + \sum_{\arc\in\route} \connectionCost[\arc].
\end{equation}

\subsection{Delay costs}\label{subsec:delay-costs}
In addition to operational costs, we also want to consider and optimize the delay costs.
For this we introduce the delay random variables $\delay$.
The \emph{departure delay} and \emph{arrival delay} of activity $\activity\in\activitySchedule$ are respectively denoted by $\legDelayDep{\activity}$ and  $\legDelayArr{\activity}$.
Because delay propagates differently depending on aircraft routes considered, these random variables heavily depend on the tail assignment solution, and are therefore conditional distributions given a feasible decision $\decision\in\feasibleDecisions{\instance}$:
\begin{equation}
    \delay\sim p(\cdot|\decision).
\end{equation}

After observing delays, the total delay cost is computed using a delay cost function $\delayCostFunction{}\colon\bbR\to\bbR^+$.
We assume in the rest of the paper that this function is piecewise linear and non-decreasing.
$\delayCostFunction{}$ is entirely defined by its $J$ slopes denoted by $(\beta_0=0, \beta_1, \cdots, \beta_J)$, and its breakpoints, denoted by $(\alpha_1=0, \cdots, \alpha_J, \alpha_{J+1}=+\infty)$, such that:
\begin{equation}
    \left\{\begin{aligned}
        &\alpha_j < \alpha_{j+1},\, \forall j\in [J],\\
        &\beta_{j-1} < \beta_{j},\, \forall j\in [J],\\
        &\delayCostFunction{}(x) = 0, \forall x \leq 0,\\
        &\delayCostFunction{}(x) = \delayCostFunction{}(\alpha_j) + (x - \alpha_j)\beta_j, \qquad \forall j\in [J],\,\forall x\in [\alpha_j, \alpha_{j+1}].
    \end{aligned}\right.
\end{equation}
This piecewise linear structure reflects how airlines view the practical impact of flight delays in operations. The increasing slopes model the disproportionate costs that longer delays incur: short delays might cause minimal disruption, while longer delays trigger larger consequences such as missed connections, passenger compensation requirements, crew overtime, etc.
This structure incentivizes the optimization model to prioritize preventing high delays that are particularly costly to airline operations and passenger satisfaction.

\begin{remark}
    For the purposes of this study, we restrict our delay cost calculations to arrival delays of flight legs only, excluding maintenance activities. This limitation stems from data availability constraints rather than methodological restrictions. The mathematical frameworks and algorithms presented in this paper are readily extensible to incorporate both maintenance delay costs and departure delay penalties without structural modifications. Such extensions would be straightforward implementation considerations when the appropriate data becomes available.
\end{remark}

\subsection{Stochastic tail assignment}
The goal of the stochastic tail assignment problem is to find a feasible solution that minimizes the total operational cost and delay cost in expectation:
\begin{equation}\label{eq:stochastic-tail-assignment}
    \min_{\decision\in\feasibleDecisions{\instance}} \underbrace{\sum_{\route\in\decision}\routeOperationalCost{\route}}_{\text{operational cost}} + \underbrace{\bbE_{\delay\sim p(\cdot|\decision)}\left[\sum_{\leg\in\legSet} \delayCostFunction{}(\legDelayArr{\leg})\right]}_{\text{delay cost}}.
\end{equation}

\begin{remark}
    While this paper focuses on the Air France tail assignment variant of the stochastic aircraft routing problem, the presented problem formulations and algorithms can easily be adapted to other variants of the problem with additional or different operational constraints.
\end{remark}

\section{Delay propagation model and sample average approximation}\label{sec:delay-model}
We present in this section a simple delay propagation model to compute activity delays from exogenous scenarios, and approximate delay costs through a scenario-based sample-average approximation.

\subsection{Measuring delays}
We call \emph{event} the departure or arrival of a leg.
We recall that each leg $\leg$ in an activity schedule $\activitySchedule$ has a \emph{scheduled departure time} $\legScheduledDep{\leg}$ and \emph{scheduled arrival time} $\legScheduledArr{\leg}$.
These two values are fixed and known constants in our problem.
Additionally, we denote by $\legObservedDep{\leg}$ and $\legObservedArr{\leg}$ the \emph{observed departure time} and \emph{observed arrival time} of the leg, respectively.
These two values may differ from the scheduled times because of random delays occurring and propagating along the schedule.
They are random variables, whose values are known only after the corresponding event occurs.

\paragraph{}
We define the \emph{departure delay} $\legDelayDep{\leg}$ of a leg $\leg$ as the difference between its \emph{observed departure time} and its \emph{scheduled departure time}:
\begin{equation}\label{eq:departure-delay}
    \legDelayDep{\leg} = \legObservedDep{\leg} - \legScheduledDep{\leg}.
\end{equation}
Similarly, we define the \emph{arrival delay} $\legDelayArr{\leg}$ of a leg $\leg$ as the difference between its \emph{observed arrival time} and its \emph{scheduled arrival time}:
\begin{equation}\label{eq:arrival-delay}
    \legDelayArr{\leg} = \legObservedArr{\leg} - \legScheduledArr{\leg}.
\end{equation}
Note that when a leg takes-off or lands early, this translates into a corresponding negative delay value.

\subsection{Delay propagation equations}

There are two kinds of delay propagation: the propagation from the departure to the arrival of the same leg, called \emph{departure-arrival propagation}, and the propagation between two successive legs, called \emph{arrival-departure propagation}.
Both propagation mechanisms can be described by simple propagation equations which link departure and arrival delays along the schedule.

\subsubsection{Departure-arrival propagation}
Delay propagation between the departure and arrival of a leg $\leg$ is the simplest one.
The observed arrival time $\legObservedArr{\leg}$ is equal to the observed departure time $\legObservedDep{\leg}$ plus the observed leg duration, which corresponds its scheduled duration plus some random perturbation $\rootDelayArr{\leg}$ we call \emph{intrinsic arrival delay}:
\begin{equation}
    \legObservedArr{\leg} = \legObservedDep{\leg} + \overbrace{\underbrace{(\legScheduledArr{\leg} - \legScheduledDep{\leg})}_{\text{scheduled leg duration}} + \rootDelayArr{\leg}}^{\text{observed leg duration}} = \legScheduledArr{\leg} + \legDelayDep{\leg} + \rootDelayArr{\leg}.
\end{equation}
We therefore obtain the departure-arrival delay propagation equation along leg $\leg$:
\begin{equation}\label{eq:delay-model/departure-arrival-propagation}
    \legDelayArr{\leg} = \underbrace{\legDelayDep{\leg}}_{\text{propagated delay}} + \underbrace{\rootDelayArr{\leg}}_{\text{intrinsic delay}}.
\end{equation}
The total observed arrival delay can be decomposed into a sum of two terms, the propagated delay from the departure event, and the root delay generated during the flight itself.
Typically, bad weather conditions or technical issues may cause a leg (which would have been on-time otherwise) to be late, independently of its propagated delay.

\subsubsection{Arrival-departure propagation}
We define the \emph{scheduled turn time} $\scheduledTurnTime{\leg}$ of a leg $\leg$ as the time needed to be ready to take-off from the previous leg arrival time.
Airlines often refers to it as \emph{minimum turn time} because it is the usual turn time we can expect when there is no schedule perturbation.
However, in practice, this value is often overestimated and observed turn times can be lower than these reference values defined as ``minimum'' turn times.
For instance, Air France computes them rather as median values than minimum ones.
The turn time is \emph{leg dependent}, because it depends among others on the destination airport and aircraft type.
We also define the \emph{slack} $\slack{\leg^-}{\leg}$ between two consecutive legs $\leg^-$ and $\leg$ as the time margin between the scheduled departure time and the scheduled \emph{ready time} $\legScheduledArr{\leg^-} + \scheduledTurnTime{\leg}$:
\begin{equation}\label{eq:delay-model/slack}
    \slack{\leg^-}{\leg} = \legScheduledDep{\leg} - \legScheduledArr{\leg^-} - \scheduledTurnTime{\leg}.
\end{equation}
It can be interpreted as the maximum amount of delay which can be absorbed by the connection between legs $\leg^-$ and $\leg$ (see Figure~\ref{fig:arrival-departure-propagation}).

Once the observed arrival time $\legObservedArr{\leg^-}$ of leg $\leg^-$ is known, the scheduled departure time of the next leg $\leg$ is adjusted to $\max(\legScheduledDep{\leg}, \legObservedArr{\leg^-} + \scheduledTurnTime{\leg}$).
The observed departure time $\legObservedDep{\leg}$ of leg $\leg$ is equal to this adjusted departure time, plus the \emph{intrinsic departure delay}:
\begin{align}
    \legObservedDep{\leg} &= \overbrace{\max\left(\legScheduledDep{\leg}, \legObservedArr{\leg^-} + \scheduledTurnTime{\leg}\right)}^{\text{minimum time the leg can depart}} + \rootDelayDep{\leg},\\
    &= \max\left(\legScheduledDep{\leg}, \legDelayArr{\leg^-} + \legScheduledArr{\leg^-} + \scheduledTurnTime{\leg}\right) + \rootDelayDep{\leg}.
\end{align}
\begin{equation}\label{eq:delay-model/arrival-departure-propagation}
    \legDelayDep{\leg} = \rootDelayDep{\leg} + \underbrace{\max(\legDelayArr{\leg^-} - \slack{\leg^-}{\leg}, 0)}_{\text{propagated delay}}.
\end{equation}
Again, the total observed departure delay is the sum of the propagated delay from the previous legs in the aircraft route, and the intrinsic departure delay $\rootDelayDep{\leg}$, which can be interpreted as the random delay perturbation accumulated between the arrival of previous leg and the departure of considered leg.

\begin{remark}
    Note that in the case there is no preceding leg, there is no propagated delay, thus equation \eqref{eq:delay-model/arrival-departure-propagation} simplifies to:
    \begin{equation}
        \legDelayDep{\leg} = \rootDelayDep{\leg}.
    \end{equation}
\end{remark}

\begin{figure}
    \centering
    \begin{subfigure}[b]{0.75\textwidth}
        \centering
        \includegraphics[width=\textwidth]{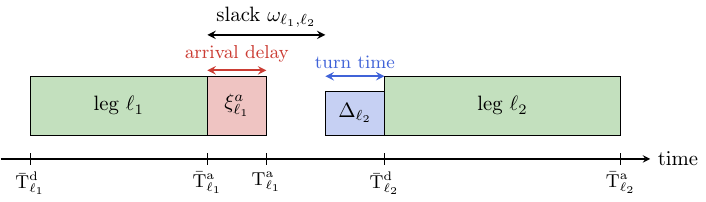}
        \caption{The arrival delay is lower than the slack, therefore the delay is fully absorbed and there is no propagated delay.}
        \label{fig:delay-prop-absorbed}
    \end{subfigure}
    
    \begin{subfigure}[b]{0.75\textwidth}
        \centering
        \includegraphics[width=\textwidth]{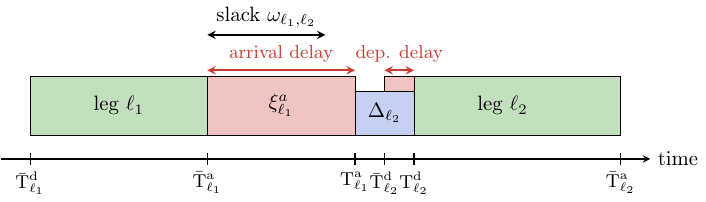}
        \caption{The arrival delay is higher than the slack, we observe in this case some propagated delay.}
        \label{fig:delay-prop-propagated}
    \end{subfigure}
    \caption{Arrival-departure propagation between two flight legs.}
    \label{fig:arrival-departure-propagation}
\end{figure}

\subsection{Sample average approximation and scenario based approach}\label{sec:delay-sample-average-approximation}
We approximate the expected delay cost in equation~\eqref{eq:stochastic-tail-assignment} by a \emph{sample average approximation} over a set of scenarios $\scenarioSet$:
\begin{equation}\label{eq:SAA}
    \bbE_{\delay\sim p(\cdot|\decision)}[\delayCostFunction{}(\delay)] \approx \frac{1}{|\scenarioSet|}\sum_{\scenario\in\scenarioSet} \delayCostFunction{}(\delay^\scenario).
\end{equation}

For each scenario $s\in\scenarioSet$, we assume to have access to intrinsic delay values $(\delayScenarioDep{\activity}, \delayScenarioArr{\activity})$ for each activity $\activity\in\activitySchedule$.
These precomputed root delays are routing-independent by construction, allowing them to be used as fixed parameters in the mathematical programming formulations. 
This enables direct computation of propagated leg delays $\delay_\leg^\scenario$ and associated costs within the optimization model.
For simpler notations in what follows, we denote by $\delayScenario{\activity}$ the sum of both departure and arrival intrinsic delays:
\begin{equation}
    \delayScenario{\activity} = \delayScenarioDep{\activity} + \delayScenarioArr{\activity}.
\end{equation}

\section{Column generation}\label{sec:column-generation}

In this section, we detail our column generation algorithm to solve the stochastic tail assignment problem.

\subsection{Set-partitioning formulation}

The stochastic tail assignment problem can be formulated as a mixed integer program using the following set-partitioning formulation:
\begin{subequations}\label{eq:column-generation}
    \begin{align}
        \min_y\quad & \sum_{\aircraft\in\aircraftSet}\sum_{\route\in\feasibleRouteSet{\aircraft}} c_{\route}^\aircraft y_\route^\aircraft\\
        \text{s.t.}\quad & \sum_{\aircraft\in\aircraftSet}\sum_{\route\ni\leg, \route\in\feasibleRouteSet{\aircraft}} y_\route^\aircraft = 1, & \forall\leg\in\legSet\label{eq:col1}\\
        & \sum_{\route\in\feasibleRouteSet{\aircraft}} y_\route^\aircraft = 1, & \forall\aircraft\in\aircraftSet\label{eq:col2}\\
        & y_\route^\aircraft\in\binarySet, & \forall\aircraft\in\aircraftSet,\,\forall\route\in\feasibleRouteSet{\aircraft}.
    \end{align}
\end{subequations}

Binary decision variables $y$ are indexed by aircraft $\aircraft\in\aircraftSet$ and feasible routes $\route\in\feasibleRouteSet{\aircraft}$ for aircraft $\aircraft$.
They are equal to 1 if route $\route$ is assigned to aircraft $\aircraft$, and 0 otherwise.
We denote by $c_\route^\aircraft$ the cost of a route $\route$ for aircraft $\aircraft$, i.e., its total operational cost plus the expected delay cost of the route, that can be computed using the sample average approximation of Section~\ref{sec:delay-sample-average-approximation}.
Constraint~\eqref{eq:col1} ensures that all legs are operated exactly once, while constraint~\eqref{eq:col2} ensures that each aircraft is used at most once.
The number of routes being exponentially large, the same is true for the number of variables.

\subsection{Column generation for the linear relaxation}
A common way to solve the linear relaxation of \eqref{eq:column-generation} is to use a column generation algorithm~\parencite{desaulniersColumnGeneration2006,uchoaOptimizingColumnGeneration2024}, with a constrained shortest path problem as pricing sub-problem.
In our case, column costs being computed using non-linear propagation equations makes solving the pricing sub-problem challenging.

Column generation can easily be implemented in practice with row generation applied to its dual formulation.
We denote by $(\lambda_\leg)_{\leg\in\legSet}$ dual variables associated with constraint~\eqref{eq:col1} and by $(\mu_\aircraft)_{\aircraft\in\aircraftSet}$ dual variables associated with constraint~\eqref{eq:col2}.
The dual of the linear relaxation of \eqref{eq:column-generation} is then:
\begin{subequations}\label{eq:dual}
    \begin{align}
        \max_{\lambda,\mu}\quad & \sum_{\leg\in\legSet}\lambda_\leg + \sum_{\aircraft\in\aircraftSet}\mu_\aircraft \\
        \text{s.t.}\quad & c_\route^\aircraft - \sum_{\leg\in\route}\lambda_\leg - \mu_\aircraft \geq 0 & \forall\aircraft\in\aircraftSet,\,\forall\route\in\feasibleRouteSet{\aircraft} \\
        & \lambda_\leg \in\bbR & \forall\leg\in\legSet\\
        & \mu_\aircraft\in\bbR & \forall\aircraft\in\aircraftSet
    \end{align}
\end{subequations}
The full column generation algorithm is detailed in Algorithm~1.

\begin{algorithm}[ht]
    \caption{Column generation algorithm}\label{algo:column-generation}
    \KwData{An instance $\instance$, a scenario set $\scenarioSet$, and for each aircraft $\aircraft\in\aircraftSet$ a set of feasible routes $\feasibleRouteSubSet{\aircraft}\subset\feasibleRouteSet{\aircraft}$.}
    \While{true}{
        Solve~\eqref{eq:dual} restricted to routes in $\feasibleRouteSubSet{\aircraft}$ for each aircraft $\aircraft$\;
        Denote by $c^{\text{low}}$ its optimal value\;

        \For{$\aircraft\in\aircraftSet$}{
            Find a route $\route$ with minimal reduced cost $\hat c_\route^\aircraft = c_\route^\aircraft - \sum_{\leg\in\route}\lambda_\leg + \mu_\aircraft$ is minimal (pricing sub-problem, see Section~\ref{subsec:pricing-sub-problem})\;
            \If{$\hat c_\route^\aircraft \geq 0$}{
                add $\route$ to $\feasibleRouteSubSet{\aircraft}$\;
            }
        }
        \If{No route was added to any route subset}{
            break\;
        }
    }
    \KwResult{Final route subsets $\feasibleRouteSubSet{\aircraft}$, dual variable values $\lambda$, $\mu$, and relaxation value $c^{\text{low}}$}
\end{algorithm}

It takes as input the instance $\instance$, the set of scenarios $\scenarioSet$, and an initial set of route $(\feasibleRouteSubSet{\aircraft})_{\aircraft\in\aircraftSet}$.
An easy way to initialize the route subsets $(\feasibleRouteSubSet{\aircraft})_{\aircraft\in\aircraftSet}$ is to find a feasible tail assignment solution, and initialize each of the route subsets with the corresponding routes found in the deterministic solution.
To find such a solution, we can for instance solve the deterministic version (see \eqref{eq:deterministic-MILP} in Appendix~\ref{appendix:compact-mip-formulation}), or run it with no objective.

\subsection{Pricing subproblem}\label{subsec:pricing-sub-problem}
In order to solve the pricing problem in Algorithm~1, we need to solve the following sub-problem for each aircraft $\aircraft$:
\begin{equation}\label{eq:pricing-sub-problem}
    \min_{\route\in\feasibleRouteSet{\aircraft}} c_\route^\aircraft - \sum_{\leg\in\route}\lambda_\leg - \mu_\aircraft
\end{equation}
It can be solved by modeling it as a  resource constrained shortest path problem on the connection sub-graph $\disubgraph{\aircraft}$.
We use a bidirectional enumeration algorithm~\parencite{irnichShortestPathProblems2005}, using dominance on forward paths and bounds on  backward paths to discard partial route solutions.
Our algorithm to solve~\eqref{eq:pricing-sub-problem} for a fixed aircraft $\aircraft$ is given by Algorithm~2.

\SetKw{KwInit}{Initialization:}
\begin{algorithm}[ht]
    \caption{Enumeration algorithm with bounding for the pricing sub-problem}\label{algo:constrained-shortest-path}
    \KwData{An acyclic digraph $\digraph = (\vertexSetBar, \arcSet)$, forward resource extension functions $\forwardExtensionFunction{\arc}$ for each arc $\arc\in\arcSet$, and bounds $b_\vertex$ for each vertex $\vertex\in\vertexSetBar$.}
    \KwInit{$\route^\star \leftarrow nothing$ and $c^\star\leftarrow+\infty$.
    $L\leftarrow\{(\source)\}$.
    $M_\source \leftarrow \{\forwardResource{\source}\}$
    $\forall \vertex\neq\source,\,M_\vertex = \emptyset$}\;
    \While{$L\neq\emptyset$}{
        $\partialPath\leftarrow$ $\source-\vertex$ path in $L$ of minimum $\bidirectionalCost(\forwardResource{\partialPath}, b_\vertex)$\;
        \For{$w\in\outneighbors{\vertex}$}{
            $\arc \leftarrow (v, w)$\;
            $\partialPath' \leftarrow \partialPath + \arc$\;
            $\forwardResource{\partialPath'}\leftarrow\forwardExtensionFunction{\arc}(\forwardResource{\partialPath})$\;
            \eIf{$w = \sink$ and $C_{\partialPath'}<c^\star$}{
                $c^\star\leftarrow C_{\partialPath'}$\;
                $\route^\star\leftarrow \partialPath'$\;}{
                \If{$\forwardResource{\partialPath'}$ is not dominated by a path in $M_w$ and $\bidirectionalCost(\forwardResource{\partialPath'}, b_w) < c^\star$}{
                    Add $\forwardResource{\partialPath'}$ to $M_w$\;
                    Remove all dominated paths from $M_w$\;
                    Add $\partialPath'$ to $L$\;
                }
            }
        }
    }
    \KwResult{Optimal route $\route^\star$ of minimum cost $c^\star$}
\end{algorithm}
For this algorithm to be valid, we need to define a partially ordered set $\forwardResourceSet$ of \emph{forward resources} defined on the graph vertices and store the state of a partial forward path starting at $\source$ ending on the considered vertex.
This set being partially ordered, we can use this order to discard partial paths that are dominated by other partial paths during the enumeration algorithm.
We also need \emph{forward extension functions} $\forwardExtensionFunction{\arc}$ defined on arcs $\arc$ that are able to extend a partial path by updating its forward resource after adding this new arc to it.
Additionally, we defined a \emph{backward resource} set $\backwardResourceSet$ and \emph{backward extension functions} $\backwardExtensionFunction{\arc}$ that are useful to encode and extend partial backward paths starting at a considered vertex $\vertex$ and ending at $\sink$.
In addition to dominance for forward paths, backward paths can be used to precompute backward resource bounds $\backwardBound{\vertex}$ at each vertex $\vertex$, which allows computing cost bounds for any given partial forward path using a bidirectional cost function $\bidirectionalCost$.
A partial path $\forwardResource{\vertex}$ can then be discarded using bounds if the cost lower bound  $\bidirectionalCost(\forwardResource{\vertex}, \backwardBound{\vertex})$ is not good enough with respect to the best known solution cost.

The algorithm maintains a list $L$ of partial paths, and a set $M_\vertex$ of non-dominated forward resources associated to partial $\source-\vertex$ paths.
At each iteration, it selects the partial path $\partialPath$ in $L$ of minimum bidirectional cost, and extends it to all its out-neighbors.
If the extension is feasible and not dominated by a path in $M_w$, we can use the precomputed bounds to discard it if its lower bound cost is greater than the current best solution.
Then, if the extension was not discarded by the bounding, we check if it is dominated by a path in $M_w$.
If it is not, we found a new non-dominated promising path, and add it to $M_w$ and $L$.
This combination of cutting using dominance and bounding allows for efficient enumeration of the paths.

What is new in this algorithm is the way bounds are computed on backward paths.
Indeed, we use the framework from~\cite{parmentierAlgorithmsNonLinearStochastic2017} and compute bound on backward resources that are convex piecewise linear functions.
In what follows, we detail the definitions forward and backward resource sets, the forward extension functions, and the bidirectional cost function used in Algorithm~2, and how to compute bounds efficiently on backward resources.

\begin{remark}
    Maintenance constraints being directly encoded into connection sub-graphs $(\disubgraph{\aircraft})_{\aircraft\in\aircraftSet}$, we do not need to check for feasibility of partial paths, as all partial paths are feasible by construction.
    This greatly simplifies the algorithm and makes it more efficient.
    Additionally, maintenances make the sub-graphs quite sparse, which can greatly accelerate the enumeration algorithm in practice.
\end{remark}

\subsubsection{Forward resources}
We define the \emph{forward resource set} as $\forwardResourceSet = \bbR^{|\scenarioSet|+1}$.
A forward resource $\forwardResource{\partialPath} = (\partialReducedCost{\partialPath}, \propagatedDelay{\partialPath})$ at vertex $\vertex$ associated to partial route $\partialPath$ as a path from $\source$ to $\vertex$: $\partialPath = (\source, \cdots, \vertex)$.
It contains the following information:
\begin{itemize}
    \item $\partialReducedCost{\partialPath}\in\bbR$ is the accumulated reduced cost of the partial route, i.e. the operational cost plus delay cost of the partial route minus the sum of the dual variables of all legs in $\partialPath$: $\partialReducedCost{\partialPath} = c_\partialPath^\aircraft - \sum_{\leg\in\partialPath} \lambda_\leg$
    \item $\propagatedDelay{\partialPath}$ is a vector of size $|\scenarioSet|$, containing the propagated delay along partial route $\partialPath$ before the activity corresponding to vertex $\vertex$ for each scenario $\scenario\in\scenarioSet$.
\end{itemize}
For each arc $\arc = (\vertex, \vertexAlt)$ in the connection graph, we define a \emph{forward extension function} $\forwardExtensionFunction{\arc}\colon \forwardResourceSet\to\forwardResourceSet$.
As we work on the sub-graph $\disubgraph{\aircraft}$, the extended partial path $\partialPath' = \partialPath + \arc = (\source, \cdots, \activity, \vertexAlt)$ is always feasible:
The extended path resource is equal to:

\begin{equation}
    \forwardExtensionFunction{\arc}(\forwardResource{\partialPath}) = 
    \left(\begin{array}{l}
        \partialReducedCost{\partialPathAlt} = \partialReducedCost{\partialPath} + \legCost[\vertex] + \connectionCost + \frac{1}{|\scenarioSet|}\sum_{\scenario\in\scenarioSet} \delayCostFunction{}\big(\delayScenario{\activity} + \propagatedDelay{\partialPath}\big) - \lambda_\vertex\oneindicator_{\{\vertex\in\legSet\}}\\[0.6ex] 
        \propagatedDelay{\partialPathAlt} = [\max\big(\delayScenario{\activity} + \propagatedDelay{\partialPath} - \slack{\vertex}{\vertexAlt}, 0\big)]_{\scenario\in\scenarioSet}.
    \end{array}\right)
\end{equation}

The arrival delay of activity $\vertex$ being equal to the propagated delay plus the root delays of the activity ($\legDelayArr{\vertex} = \delayScenario{\activity} + \propagatedDelay{\partialPath}$), the reduced cost is extended by adding to it the operational cost of activity $\vertex$, the connection cost of arc $\arc$, the mean delay costs $\delayCostFunction{}(\legDelayArr{\vertex})$ over scenarios, and subtracting the dual variable of the vertex $\lambda_\vertex$ (if it is a leg).
Then, we update the propagated delay in each scenario by subtracting the slack (see equation~\eqref{eq:delay-model/arrival-departure-propagation}), and update the maintenance index if the vertex is a maintenance activity.
Following this definition, we therefore have $\forwardResource{\partialPath'} = \forwardExtensionFunction{\arc}(\forwardResource{\partialPath})$.
The forward resource at the empty partial path is defined by $\{\source\}$ is $\forwardResource{\source} = (0, (0)_{\scenario\in\scenarioSet}, 1)$.

Additionally, we define a partial order $\preceq$ on the forward resource set.
Then, for any two forward resources $\forwardResource{\partialPath_1}$ and $\forwardResource{\partialPath_2}$ associated to two partial paths $\partialPath_1$ and $\partialPath_2$ ending at the same vertex $\activity$, we have:
\begin{equation}
    \forwardResource{\partialPath_1}\preceq\forwardResource{\partialPath_2} \text{ if and only if }
    \begin{cases}
        \forall\scenario\in\scenarioSet,\,\propagatedDelay{\partialPath_1,\scenario}\leq\propagatedDelay{\partialPath_2,\scenario}\\
        \partialReducedCost{\partialPath_1} \leq \partialReducedCost{\partialPath_2}
    \end{cases}
\end{equation}
We say that an $\source-\vertex$ path $\partialPath_1$ \emph{dominates} another $\source-\vertex$ path $\partialPath_2$ if $\forwardResource{\partialPath_1}\preceq\forwardResource{\partialPath_2}$.
For the enumeration algorithm to be able to use dominance to discard paths, we need to ensure that the forward extension function is non-decreasing with respect to this partial order $\preceq$. 
It is the case by definition of those extension functions.
Under this assumption, the following proposition shows dominated paths can safely be pruned during the enumeration algorithm.
\begin{proposition}
    Let $\route = \partialPath_1 + \partialPath_2$ an optimal solution of~\eqref{eq:pricing-sub-problem}, with $\partialPath_1$ an $\source-\vertex$ path and $\partialPath_2$ a $\vertex-\sink$ path.
    If $\partialPath'_1$ is an $\source-\vertex$ path such that $\forwardResource{\partialPath'_1}\preceq\forwardResource{\partialPath_1}$, then $\route' = \partialPath'_1 + \partialPath_2$ is an optimal solution of~\eqref{eq:pricing-sub-problem}.
\end{proposition}

\begin{proof}{Proof}
    $\partialPath_1$ is feasible and $\forwardResource{\partialPath'_1}\preceq\forwardResource{\partialPath_1}$, therefore $\partialPath'_1$ is feasible.
    Then, resource extension functions being non-decreasing, we have $\forwardResource{\route'}\preceq\forwardResource{\route}$, therefore $\route'$ is also feasible and of reduced cost non-greater than $\route$: $C_{\route'}\leq C_\route$.
\end{proof}

\subsubsection{Backward resources}
We denote the \emph{backward resource set} by $\backwardResourceSet$.
A \emph{backward resource} $\backwardResource{\partialPath} = (\,\partialBackwardReducedCost{\partialPath}, \backwardFunction{\partialPath}, \lambda_\partialPath)\in\backwardResourceSet$ is associated to a partial $\vertex-\sink$ path.
It contains the following information:
\begin{itemize}
    \item $\partialBackwardReducedCost{\partialPath}\in\bbR$ is the accumulated operational cost of the partial path.
    \item $\backwardFunction{\partialPath}$ is a vector containing $|\scenarioSet|$ piecewise linear functions of size, one for each scenario $\scenario\in\scenarioSet$.
    Each component $\backwardFunction{\partialPath}^\scenario\colon\bbR\to\bbR$, takes as input the propagated delay $\propagatedDelay{\vertex}$ before vertex $\vertex$ and returns the total delay cost of the partial path $\partialPath$ for scenario $\scenario$.
    \item $\lambda_\partialPath$ is the accumulated dual variables summed over the legs in the partial route: $\lambda_\partialPath = \sum_{\leg\in\partialPath\cap\legSet}\lambda_\leg$.
\end{itemize}
For each arc $\arc = (\vertexu, \vertex)$ in the connection graph, we define the corresponding \emph{backward extension function} $\backwardExtensionFunction{\arc}\colon\backwardResourceSet\to\backwardResourceSet$.
We extend a given backward path as follows:
\begin{equation}
    \backwardExtensionFunction{\arc}(\,\backwardResource{\partialPath}) = 
    \left(\begin{aligned}
        &\partialBackwardReducedCost{\partialPathAlt} = \partialBackwardReducedCost{\partialPath} + \legCost[\vertex] + \connectionCost\\
        &\forall\scenario\in\scenarioSet,\, \backwardFunction{\partialPathAlt}^\scenario = \delayCostFunction{} \circ \arrivalDelayFunction{\arc} + \backwardFunction{\partialPath}^\scenario \circ \delayPropagationFunction{\arc}\\
        &\lambda_{\partialPathAlt} = \lambda_{\partialPath} + \lambda_w
    \end{aligned}\right),
\end{equation}
with $\arrivalDelayFunction{\arc}: x\mapsto \delayScenario{\vertexu} + x$ the arrival delay function, and $\delayPropagationFunction{\arc}: x\mapsto \delayScenario{\vertexu} + \max(x - \slack{\vertexu}{\vertex}, 0)$ the delay propagation function, along arc $\arc$.
The backward operational cost is easily extended by adding the operational cost of vertex $\vertex$ and the connection cost of arc $\arc$.
The dual variable sum is extended similarly.
Extending the delay cost function is more complex.
It is defined as a sum of two functions: a function that computes the delay cost of vertex $\vertex$, and a function that computes the delay cost of the rest of the path.
The first term if computed by composing the delay cost function $\delayCostFunction{}$ with the arrival delay function that computes the arrival delay of $\vertexu$ from the propagated delay before $\vertexu$, while the second term is computed by composing the delay cost function with the delay propagation function that computes the propagated delay before vertex $\vertex$ from the propagated delay before $\vertexu$.
Note that, after extension, components of the resulting vector $\backwardFunction{\partialPathAlt}$ are convex and piecewise linear functions as a sum and composition of convex piecewise linear functions.
Similarly to the forward resources, this definition therefore gives us $\backwardResource{\partialPathAlt} = \backwardExtensionFunction{\arc}(\,\backwardResource{\partialPath})$.
The backward resource at the sink $\sink$ is defined as $\backwardResource{\sink} = (0, (x\mapsto 0)_{\scenario\in\scenarioSet}, 0)$.

We can also define a partial order $\preceq$ on the backward resource set.
For any two forward resources $\backwardResource{\partialPath_1}$ and $\backwardResource{\partialPath_2}$ associated to two partial paths $\partialPath_1$ and $\partialPath_2$ starting at the same vertex $\vertex$, we have:
\begin{equation}
    \begin{aligned}
        &\backwardResource{\partialPath_1}\preceq\backwardResource{\partialPath_2} \text{ if and only if: }\\
        &\begin{cases}
            \forall\scenario\in\scenarioSet,\forall x\in\bbR,\, \backwardFunction{\partialPath_1}^\scenario(x) \leq \backwardFunction{\partialPath_2}^\scenario(x) \\
            \partialBackwardReducedCost{\partialPath_1} \leq \partialBackwardReducedCost{\partialPath_2}\\
            \lambda_{\partialPath_1} \geq \lambda_{\partialPath_2}
        \end{cases}
    \end{aligned}
\end{equation}
We define the meet $\wedge$ over $\backwardResourceSet$ between two backward resources $\backwardResource{\partialPath}$ and $\backwardResource{\partialPathAlt}$ as the greatest lower bound:
\begin{equation}
    \backwardResource{\partialPath}\wedge\backwardResource{\partialPathAlt} = 
    \left(\begin{aligned}
        &\forall\scenario\in\scenarioSet,\, \backwardFunction{\partialPath}^\scenario \wedge\backwardFunction{\partialPathAlt}^\scenario\\
        &\min(\,\cev C_{\partialPath}, \,\cev C_{\partialPathAlt})\\
        &\max(\lambda_{\partialPath}, \lambda_{\partialPathAlt})
    \end{aligned}\right)
\end{equation}
with $f\wedge g$ for $f$ and $g$ piecewise linear function being another piecewise linear function which is a lower bound of $f$ and $g$.
This meet operation can be used to define the generalized dynamic programming equation:
\begin{equation}
    b_\vertex = \begin{cases}
        ((0\mapsto 0)_{\scenario\in\scenarioSet}, 0, 0) &\text{ if } \vertex = \sink\\
        \bigwedge\limits_{\arc = (\vertex, w)\in\outneighbors{\vertex}} \backwardExtensionFunction{\arc}(b_w) &\text{ otherwise}
    \end{cases}
\end{equation}
\cite{parmentierAlgorithmsNonLinearStochastic2017} gives that this generalized dynamic programming equation gives a lower bound on the backward resources:
\begin{proposition}
    $b_\vertex\preceq\,\backwardResource{\partialPath}$ for any vertex $\vertex$ and all $\vertex-\sink$ paths $\partialPath$.
\end{proposition}
\noindent This lower bound can easily be computed using a Ford-Bellman algorithm on the acyclic connection digraph.

\begin{remark}
    For a given aircraft $\aircraft$, the only thing that changes in the pricing sub-problem~\eqref{eq:pricing-sub-problem} between two column generation iterations are the dual variables $\lambda_\leg$ and $\mu_\aircraft$.
    Therefore, in practice, most components of the bounds can be precomputed before the column generation algorithm starts, and only the component related to the dual variables need to be recomputed at each iteration.
    In particular, the piecewise linear functions of the backward resources can be precomputed once for all before the column generation algorithm starts, which saves a lot of computation time.
\end{remark}

\subsubsection{Bidirectional cost function}
We define the bidirectional cost function $\bidirectionalCost:\forwardResourceSet\times\backwardResourceSet\to\bbR\cup\{+\infty\}$.

If two paths, $\partialPath$ and $\partialPathAlt$, are such that $\partialPath$ is a $\source-\vertex$ path and $\partialPathAlt$ is a $\vertex-\sink$ path for a vertex $\vertex$, the bidirectional cost of two associated forward and backward resources $\forwardResource{\partialPath}$ and $\backwardResource{\partialPathAlt}$ is defined as:
\begin{equation}
    \bidirectionalCost(\forwardResource{\partialPath},\,\backwardResource{\partialPathAlt}) = C_{\partialPath} + \cev C_{\partialPathAlt} + \dfrac{1}{|\scenarioSet|}\sum\limits_{\scenario\in\scenarioSet} \backwardFunction{\partialPathAlt}^\scenario(\propagatedDelay{\partialPath, \scenario}) - \lambda_{\partialPathAlt},
\end{equation}
else we set it to $+\infty$.

\subsubsection{Computing the meet of two convex piecewise linear functions}
The quality of the backward bound $\backwardBound{\vertex}$ depends on how we compute the meet of two piecewise linear functions.
The best possible lower is taking the minimum between the two functions:
\begin{equation}
    \backwardFunction{\partialPath}^\scenario \wedge\backwardFunction{\partialPathAlt}^\scenario = \min(\backwardFunction{\partialPath}^\scenario, \backwardFunction{\partialPathAlt}^\scenario)
\end{equation}
However, while this is a very good bound, in practice, on large instances, this min operation makes the number of breakpoints of resulting piecewise linear function explode.
This results in a very slow algorithm to compute the bounds.
To mitigate this issue, we recommend using a slightly worse bound by taking the convex lower bound instead of the minimum.
Indeed, if both functions are convex, we can keep the convexity of the resulting function by taking the best convex lower bound, which will not increase the number of breakpoints.
This bound is very efficient to compute while giving a good enough lower bound to prune partial paths in the enumeration algorithm using it.

\subsection{Diving heuristic}
While the column generation algorithm efficiently computes tight lower bounds for the linear relaxation of~\eqref{eq:column-generation}, obtaining high-quality integer solutions requires additional techniques. As we demonstrate in Section~\ref{sec:numerical-experiments}, using a restricted master heuristic (solving the integer program on the generated columns) is often insufficient for finding good solutions to the stochastic tail assignment problem.
To address this limitation, we implement the \emph{diving heuristic with limited backtracking} algorithm from \cite{sadykovPrimalHeuristicsBranch2019}.

The diving heuristic performs a depth-first exploration of the branch-and-price tree. At each node, we apply the column generation algorithm to obtain the optimal linear relaxation. Then, we select the route variable $y_\route^\aircraft$ with the largest fractional value and fix it to 1, effectively assigning route $\route$ to aircraft $\aircraft$.
Unlike greedy diving approaches that permanently fix variables, our implementation uses limited backtracking when a node becomes infeasible.
Specifically, if fixing a variable to 1 leads to an infeasible subproblem (no feasible integer solution exists), we backtrack to the previous node and try fixing an alternative variable with the next-largest fractional value.

This process continues until we find a feasible integer solution or exhaust the search tree within our computational limits.
To prevent cycling and ensure termination, we maintain a taboo list of route variables that have been previously fixed and led to infeasibility. The algorithm terminates when this taboo list exceeds a predefined size parameter $k\in\bbN$.
Following \cite{sadykovPrimalHeuristicsBranch2019}, we also impose a $maxDepth$ parameter that limits the maximum depth of the search tree.
In our case, we set $maxDepth = |\aircraftSet|$, since any feasible solution will contain exactly $|\aircraftSet|$ routes (one for each aircraft).

The diving approach is particularly effective for our stochastic tail assignment problem for several reasons.
First, it progressively reduces the problem size by fixing variables, which accelerates the column generation at deeper nodes.
Second, it can quickly find high-quality solutions even for instances where a restricted master heuristic struggles to converge due to the large number of generated columns.
Finally, at each step of the diving process, we can efficiently prune the connection graph by removing the aircraft associated with the fixed route and deleting all activities already covered by fixed routes, further reducing computational complexity.

\section{Numerical experiments}\label{sec:numerical-experiments}
In this section, our goal is to show that the stochastic approach indeed leads to better solution than the traditional deterministic approach, and is tractable in practice.
To that purpose, we benchmark extensively our column generation and diving heuristic on industrial instances from Air France, and show their performance in terms of quality of the solution returned, computing time, and scalability.

\subsection{Tail assignment instances}\label{subsec:air-france-datasets}

We have access to a set of tail assignment instances provided by Air France.
They are past instances from the end of 2022.
It contains 129 instances of various sizes, with a time horizon between 4 and 10 days, containing between 60 and 600 scheduled flight legs, between 6 and 55 scheduled maintenances.
An instance contains between 4 and 20 aircraft of the same aircraft type, part of one of 10 different fleets.
The associated connection graph sizes range from 1000 to 50000 arcs.
We split the instances into 3 categories based on the number of legs:
\begin{itemize}
    \item \textbf{Small instances}: 60 to 200 legs, 49 instances.
    \item \textbf{Medium instances}: 201 to 400 legs, 62 instances.
    \item \textbf{Large instances}: more than 400 legs, 18 instances.
\end{itemize}

Fuel costs and connection costs are provided for each leg and connection, as well as scheduled turn times.
The delay cost functions are also provided for each aircraft type, i.e., for each instance.
Figure~\ref{fig:tail-assignment-instance} shows an example of a feasible solution for one of the smallest instances in the dataset.

\begin{figure}[ht]
    \centering
    \includegraphics[width=0.75\textwidth]{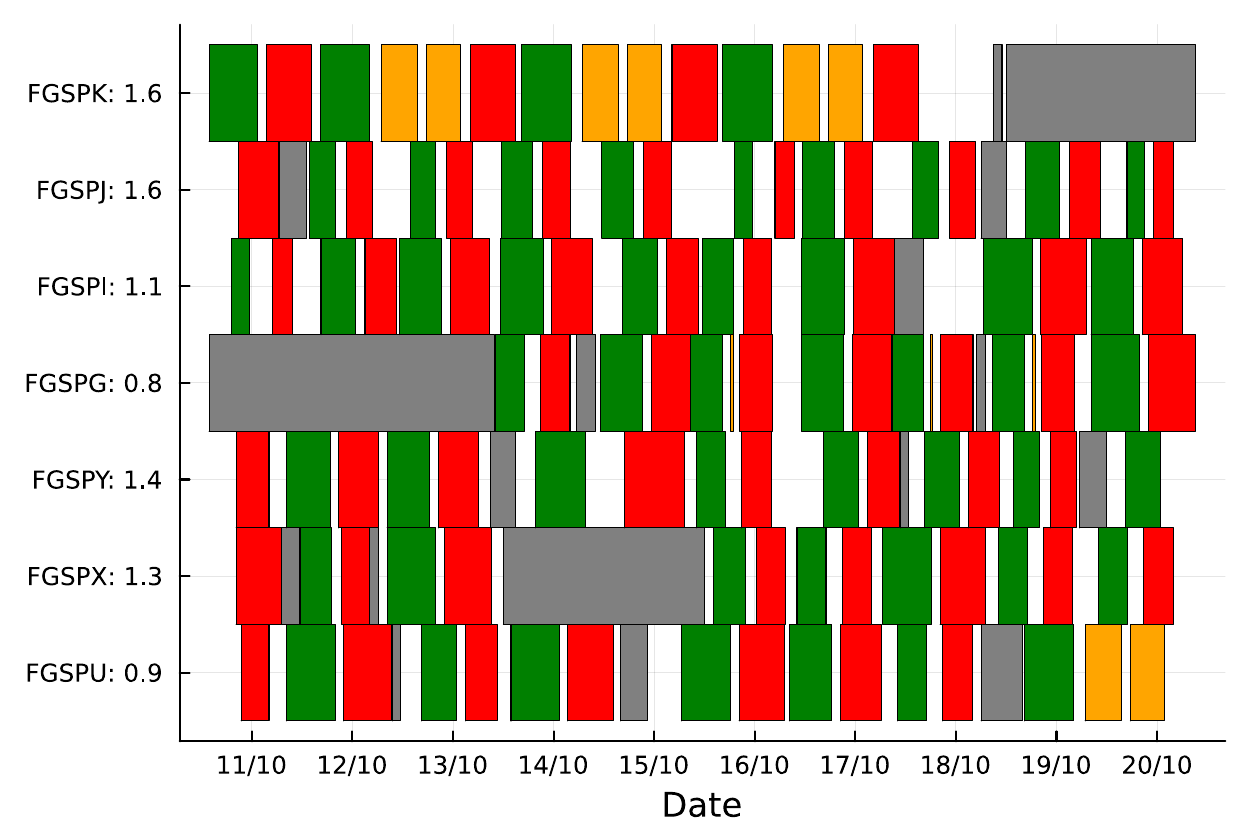}
    \caption{Example of a tail assignment instance and a feasible solution. This instance has 115 legs, 17 maintenances, and 7 aircraft. Each row correspond to an aircraft route, each box corresponding to an activity. Grey boxes are maintenances, colored boxes are flight legs. Green legs are legs that depart from the Charles de Gaulle airport, red legs are legs that arrive at the Charles de Gaulle airport, and yellow legs are legs that are not connected to the Charles de Gaulle airport.}
    \label{fig:tail-assignment-instance}
\end{figure}

\begin{remark}
    It is important to note that the instance sizes in our dataset, while ranging from 60 to 600 flight legs, are representative of real-world operational practice.
    Airlines typically solve the tail assignment problem at the sub-fleet level rather than for their entire fleet simultaneously.
\end{remark}

\subsection{Benchmark algorithms}\label{subsec:benchmark-algorithms}

To evaluate the effectiveness of our proposed column generation and diving heuristic approach, we compare against the three following benchmark algorithms.

\subsubsection{Compact MIP formulation}
We implement a compact mixed-integer programming formulation that models the stochastic tail assignment problem directly using binary variables for each aircraft-leg assignment and scenario-dependent delay propagation constraints.
This formulation serves as our primary baseline to demonstrate the necessity of the column generation approach.
It is detailed in Appendix~\ref{appendix:compact-mip-formulation}.
While this approach provides exact solutions when solved to optimality, it faces significant scalability challenges as both instance size and the number of scenarios increase, as we demonstrate in our computational results.

\subsubsection{Deterministic baseline}
As an additional baseline, we solve a deterministic version of the tail assignment problem that minimizes only operational costs, completely ignoring delay considerations.
This deterministic approach represents the traditional methodology used in many airline planning systems and allows us to quantify the value of incorporating stochastic delay information.
The deterministic model can be solved efficiently using standard MIP solvers, providing a fast but potentially suboptimal solution that serves as our baseline for cost comparisons.
It is also detailed in Appendix~\ref{appendix:compact-mip-formulation}.

This approach is the closest to what would be used in practice by airlines, the main difference being that we do not enforce turn time constraints in the deterministic model, for feasibility reasons.
Indeed, to avoid high delay costs, airlines typically enforce minimum time buffers between legs, these often leading to infeasibility and needing to be fixed by hand.
To maintain tractable benchmark instances, we deliberately omit these turn time constraints.
Consequently, our deterministic baseline achieves unrealistically low operational costs compared to practical implementations, while potentially incurring substantial delay costs due to tight scheduling.
This baseline thus serves two analytical purposes: quantifying the operational value of incorporating delay uncertainty in tail assignment, and providing a controlled comparison against traditional deterministic optimization approaches without manual solution corrections.

\subsubsection{Restricted master heuristic}
Finally, we implement a restricted master heuristic (see Appendix~\ref{appendix:restricted-master-heuristic} for details).
This approach represents a natural heuristic extension of column generation and serves to demonstrate the necessity of the more sophisticated diving heuristic.
As our results show, this approach often struggles to find high-quality integer solutions due to the limited column set and the complex structure of the stochastic problem.

\subsection{Experiments design}
On each of the 129 Air France instances, we benchmark our column generation algorithm and diving heuristic against the three baselines described above.
The optimization solution framework is implemented in the Julia programming language.
The pricing sub-problem is solved using a bidirectional enumeration algorithm implemented in our open-source package \texttt{ConstrainedShortestPaths.jl}\footnote{See \url{https://github.com/BatyLeo/ConstrainedShortestPaths.jl}}.
The piecewise linear delay cost functions and the associated bound computations are handled using the \texttt{PiecewiseLinearFunctions.jl}\footnote{See \url{https://github.com/BatyLeo/PiecewiseLinearFunctions.jl}} package, which we also developed.
While these generic building blocks are open-source, the Air France specific implementation code and the industrial datasets remain unfortunately private for confidentiality reasons.

For all experiments, we use Gurobi 12.0~\parencite{gurobioptimizationllcGurobiOptimizerReference2024} with default settings.
We set a time limit of 3600 seconds per instance for the MIP formulation and the restricted master heuristic.
Each instance is solved for $1$, $5$, $10$, $50$, and $100$ scenarios.
We use the solution of the deterministic MILP as a warm start for the MIP, and as the initial column set in the column generation algorithm.
The diving heuristic is parametrized with a taboo list size of $k=15$.

\subsection{Results}

\subsubsection{MIP formulation}

In Table~\ref{tab:MIP-performance}, we report the performance of the compact MIP formulation on the Air France instances.
\begin{table}[ht]
    \centering
    \begin{tabular}{lccc}
        \hline
        \textbf{Number of scenarios} & \textbf{Average time} & \multicolumn{1}{p{3.1cm}}{\centering \textbf{Instances solved} \\\textbf{to optimality}} & \textbf{Average gap} \\
        \hline
        \textbf{Deterministic (0 scenarios)} & 5s & 129/129 & 0\% \\
        \quad Small instances & 0.9s & 49/49 & 0\% \\
        \quad Medium instances & 5s & 62/62 & 0\% \\
        \quad Large instances & 16s & 18/18 & 0\% \\
        \hline
        \textbf{1 scenario} & 1814s & 79/129 & 11\% \\
        \quad Small instances & 100s & 48/49 & 0\% \\
        \quad Medium instances & 2636s & 31/62 & 3\% \\
        \quad Large instances & 3600s & 0/18 & 70\% \\
        \hline
        \textbf{5 scenarios} & 2346s & 53/129 & 50\% \\
        \quad Small instances & 431s & 48/49 & 0.2\% \\
        \quad Medium instances & 3495s & 5/62 & 71\% \\
        \quad Large instances & 3600s & 0/18 & 111\% \\
        \hline
        \textbf{10 scenarios} & 2545s & 51/129 & 64\% \\
        \quad Small instances & 882s & 48/49 & 1.3\% \\
        \quad Medium instances & 3553s & 4/62 & 99\% \\
        \quad Large instances & 3600s & 0/18 & 111\% \\
        \hline
    \end{tabular}
    \caption{Performance of the compact MIP formulation on the Air France instances}
    \label{tab:MIP-performance}
\end{table}
We observe that incorporating delay considerations significantly increases computational complexity.
The deterministic version, which only minimizes operational costs (which are linear), can be solved in just 5 seconds on average across all 129 instances, with even the largest instances solved within 16 seconds.

However, when we introduce just a single delay scenario, computational complexity increases.
The average solve time jumps to over 30 minutes, and only 79 of 129 instances can be solved to optimality within the one-hour time limit.
This effect is particularly pronounced for larger instances: while nearly all small instances are solved optimally with one scenario, none of the large instances reach optimality, resulting in an average gap of 70\%.

As we increase the number of scenarios, the solver's performance deteriorates.
With 5 scenarios, less than half instances are solved to optimality, and the average gap increases to 50\%.
The pattern continues with 10 scenarios, with an average gap of 64\%.

The breakdown by instance size reveals that small instances remain relatively tractable even with multiple scenarios, while medium and large instances face severe scalability issues.
On medium instances, the corresponding average gaps increase dramatically from 3\% with one scenario to 99\% with 10 scenarios.

Figure~\ref{fig:relax-gaps} further illustrates how solution quality deteriorates with problem size.
The gap to the lower bound begins increasing for all methods around 115 legs and grows exponentially for instances with more than 300 legs, even with just a single scenario.
\begin{figure}[htbp]
    \centering
    \includegraphics[width=0.7\textwidth]{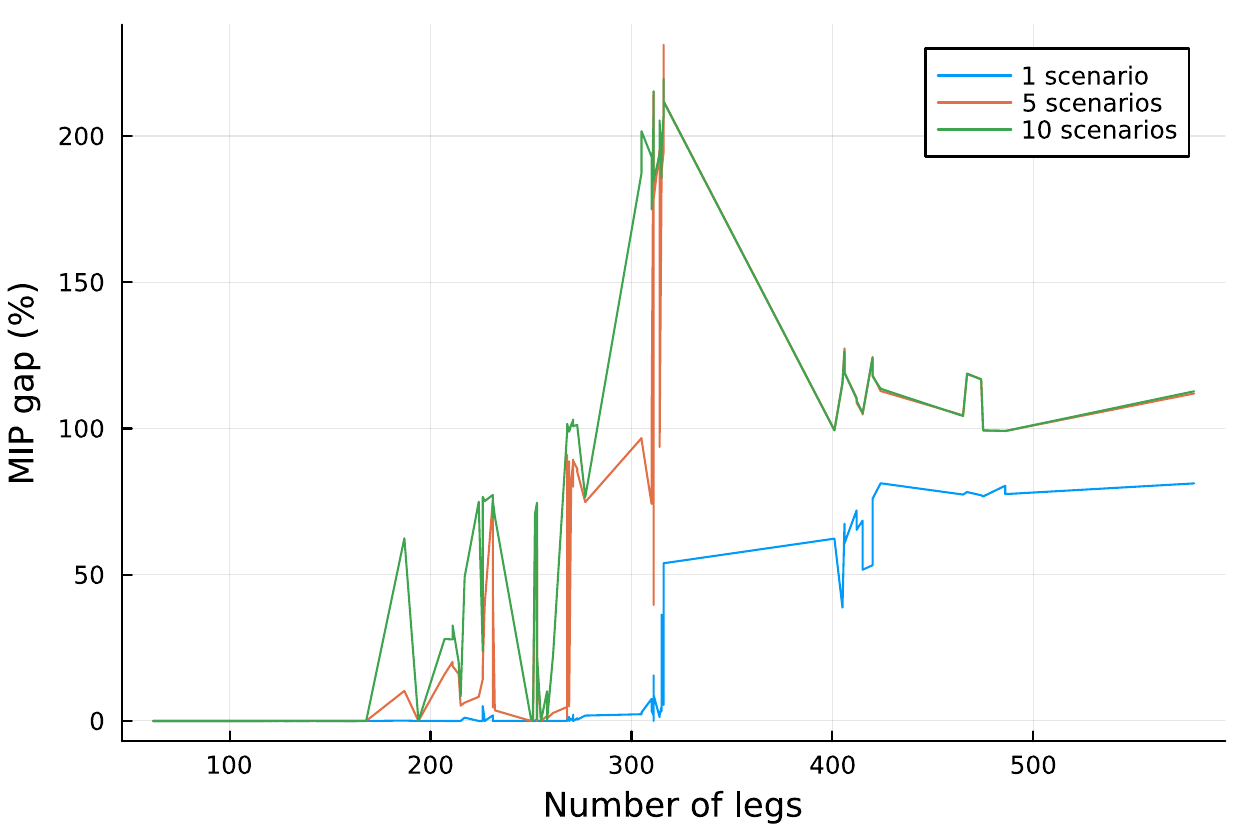}
    \caption{MIP gap values with respect to the number of legs in the instances. \label{fig:relax-gaps}}
\end{figure}

\begin{result}
    The MIP formulation faces significant scalability challenges as both the number of legs and scenarios increase.
    While effective for small instances across all scenario counts, it struggles with medium and large instances, where the average gap to the lower bound increases substantially and the proportion of instances solved to optimality quickly decreases.
\end{result}

\subsubsection{Set-partitionning formulation}

We report in Table~\ref{tab:colgen-performance} the performance of the set-partitionning formulation on the Air France instances.
We report the average and maximum time to solve the relaxation with column generation, as well as the average and maximum time and gaps of both the restricted master heuristic and diving heuristic.
Gaps are computed with respect to the lower bound obtained by the column generation algorithm.

\begin{table}[ht]
    \centering
    \begin{tabular}{clccccc}
        \hline
        \multirow{2}{*}{\textbf{Scenarios}} & \multirow{2}{*}{\textbf{Instances}} & \multicolumn{1}{c}{\textbf{Col. gen.}} & \multicolumn{2}{c}{\textbf{Restricted master}} & \multicolumn{2}{c}{\textbf{Diving heuristic}} \\
        \cline{3-7}
        & & \textbf{Time (s)} & \textbf{Time (s)} & \textbf{Gap (\%)} & \textbf{Time (s)} & \textbf{Gap (\%)} \\
        \hline
        \multirow{4}{*}{\textbf{1}} & \textbf{All} & 578 (10102) & 379 (3600) & 74 (1800) & 870 (48837) & 0.13 (8.6) \\
        & Small & 12 (175) & 1 (11) & 104 (1800) & 25 (808) & 0.19 (8.6) \\
        & Medium & 395 (1796) & 261 (3600) & 40 (631) & 332 (1855) & 0.05 (1.2) \\
        & Large & 2751 (10102) & 1812 (3600) & 110 (1656) & 5018 (48837) & 0.19 (1.7) \\
        \hline
        \multirow{4}{*}{\textbf{5}} & \textbf{All} & 667 (8725) & 654 (3600) & 137 (1485) & 1381 (81433) & 0.17 (5.1) \\
        & Small & 15 (236) & 1 (3) & 132 (89) & 20 (296) & 0.09 (3.7) \\
        & Medium & 405 (2150) & 798 (3600) & 111 (1485) & 317 (1827) & 0.2 (5.1) \\
        & Large & 3348 (8725) & 1936 (3600) & 240 (1385) & 8753 (81433) & 0.28 (2.9) \\
        \hline
        \multirow{4}{*}{\textbf{10}} & \textbf{All} & 915 (14138) & 780 (3600) & 147 (2022) & 818 (15611) & 0.42 (30) \\
        & Small & 18 (231) & 1 (4) & 140 (2022) & 26 (292) & 0.13 (5.2) \\
        & Medium & 609 (4018) & 865 (3600) & 138 (1569) & 404 (1417) & 0.69 (30) \\
        & Large & 4407 (14138) & 2608 (3600) & 168 (1245) & 4395 (15611) & 0.3 (3.1) \\
        \hline
        \multirow{4}{*}{\textbf{50}} & \textbf{All} & 1596 (23095) & 638 (3600) & 143 (2279) & 1338 (21331) & 0.21 (5.6) \\
        & Small & 35 (503) & 1 (4) & 177 (1373) & 31 (235) & 0.09 (2.9) \\
        & Medium & 833 (3855) & 918 (3600) & 127 (2279) & 676 (2381) & 0.22 (5.6) \\
        & Large & 8472 (23095) & 1409 (3600) & 108 (1164) & 7174 (21331) & 0.49 (4.4) \\
        \hline
        \multirow{4}{*}{\textbf{100}} & \textbf{All} & 2738 (43081) & 568 (3600) & 117 (1950) & 2061 (33371) & 0.28 (6.8) \\
        & Small & 56 (762) & 1 (3) & 94 (203) & 55 (482) & 0.21 (5.7) \\
        & Medium & 1518 (11560) & 772 (3600) & 132 (1950) & 1078 (8667) & 0.28 (6.8) \\
        & Large & 14242 (43081) & 1410 (3600) & 125 (1723) & 10908 (33371) & 0.47 (4.7) \\
        \hline
        \multicolumn{7}{p{\textwidth}}{\footnotesize Average value is first provided, then maximum in parenthesis. E.g., the average time for the column generation with 1 scenario on all instances is 578s and the maximum 10102s.}
    \end{tabular}
    \caption{Performance of the set-partitioning formulation on the Air France instances. Each cell contains the average and maximum values on the dataset, broken down by instance size.}
    \label{tab:colgen-performance}
\end{table}

Looking at the performance data in Table~\ref{tab:colgen-performance}, we can draw several conclusions about the different solution approaches.

The set-partitioning formulation demonstrates superior scalability compared to the compact MIP formulation, allowing it to effectively handle instances with larger numbers of legs and more scenarios.
The column generation algorithm efficiently computes lower bounds for all 129 instances across the entire range of scenarios, whereas the MIP approach struggles even with smaller scenario sets.

When examining solution times, we observe significant variation by instance size.
For column generation with a single scenario, solving small instances takes on average just 12 seconds, medium instances about 7 minutes, while large instances require an average of 45 minutes.
As expected, these times increase with the number of scenarios, reaching an average of nearly 4 hours for large instances with 100 scenarios.

The restricted master heuristic shows reasonable average solve times (6-13 minutes across all scenarios), but it consistently hits the 1-hour time limit for larger instances.
This suggests the heuristic struggles to find high-quality integer solutions within the allocated time for complex problems.

In comparison, the diving heuristic demonstrates excellent solution quality, achieving average gaps of less than 0.5\% across all scenario configurations, with the largest gap for any instance being 30\%.
This suggests that the progressive variable-fixing approach employed by the diving heuristic is particularly well-suited to this problem structure, and also confirms that the column generation algorithm is effective at providing tight lower bounds across all instances.

However, this superior solution quality comes at the cost of increased computational time compared to the restricted master heuristic, particularly for challenging instances.
The maximum solve time for the diving heuristic can exceed 9 hours for instances with 100 scenarios.

\begin{result}
    The set-partitioning formulation, demonstrates better scalability.
    The column generation algorithm efficiently computes tight lower bounds, for all instances and for more scenarios than the MIP formulation.
    However, the restricted master heuristic exhibits high gaps for some instances, which leads to high gap in average.
    The diving heuristic, on the other hand, achieves near-optimal solutions with an average gap of less than 0.5\%, at the cost of increased computational time.
\end{result}

\subsubsection{Comparison to deterministic approach}

Finally, we compare the performance of our stochastic algorithms to the deterministic approach.
We compute the average improvement of the stochastic formulations over the deterministic solution.
For this, for each solution, we compute the operational cost, the delay cost, and the total cost, with 100 delay scenarios.
We report the average improvement over the deterministic solution in Table~\ref{tab:deterministic-improvement}.
\begin{table}[ht]
    \centering
    \begin{tabular}{ccccccc}
        \hline
        \multicolumn{2}{c}{\textbf{Number of scenarios}} & \textbf{1} & \textbf{5} & \textbf{10} & \textbf{50} & \textbf{100} \\
        \hline
        \multirow{9}{*}{\textbf{Restricted master heuristic}} & \textbf{Operational cost} & 15.1\% & 14\% & 15.4\% & 16\% & 16\% \\ 
        & \quad Small & 5.1\% & 5.6\% & 5.2\% & 5.5\% & 5.5\% \\
        & \quad Medium & 26\% & 24\% & 27\% & 27\% & 27\% \\
        & \quad Large & 4.8\% & 5.6\% & 4.5\% & 7.5\% & 6.4\% \\
        \cline{2-7}
        & \textbf{Delay cost} & -66\% & -63\% & -58\% & -61\% & -64\% \\ 
        & \quad Small & -61\% & -64\% & -62\% & -62\% & -65\% \\
        & \quad Medium & -75\% & -66\% & -64.\% & -62\% & -65\% \\
        & \quad Large & -48\% & -48\% & -30\% & -58\% & -58\% \\
        \cline{2-7}
        & \textbf{Total cost} & -51\% & -48\% & -44\% & -46\% & -49\% \\
        & \quad Small & -39\% & -41\% & -39\% & -39\% & -41\% \\
        & \quad Medium & -64\% & -55\% & -53\% & -51\% & -54\% \\
        & \quad Large & -41\% & -42\% & -25\% & -51\% & -50\% \\
        \hline
        \multirow{9}{*}{\textbf{Diving heuristic}} & \textbf{Operational cost} & 19\% & 21\% & 21\% & 21\% & 21\% \\ 
        & \quad Small & 5.1\% & 5.6\% & 5.8\% & 5.8\% & 5.6\% \\
        & \quad Medium & 33\% & 36\% & 37\% & 37\% & 37\% \\
        & \quad Large & 6.6\% & 7.4\% & 7.9\% & 7.9\% & 7.9\% \\
        \cline{2-7}
        & \textbf{Delay cost} & -77\% & -79\% & -79\% & -79\% & -79\% \\ 
        & \quad Small & -67\% & -69\% & -69\% & -70\% & -70\% \\
        & \quad Medium & -82\% & -83\% & -84\% & -84\% & -84\% \\
        & \quad Large & -86\% & -88\% & -89\% & -89\% & -89\% \\
        \cline{2-7}
        & \textbf{Total cost} & -60\% & -62\% & -62\% & -62\% & -62\% \\
        & \quad Small & -43\% & -44\% & -44\% & -45\% & -45\% \\
        & \quad Medium & -70\% & -71\% & -71\% & -71\% & -71\% \\
        & \quad Large & -75\% & -77\% & -77\% & -77\% & -77\% \\
        \hline
    \end{tabular}
    \caption{Average gap to the deterministic solution broken down by instance size.}
    \label{tab:deterministic-improvement}
\end{table}

Table \ref{tab:deterministic-improvement} presents a clear trade-off between operational costs and delay costs when comparing stochastic approaches to the deterministic baseline.
Both heuristics show increased operational costs (between 14-21\%) while achieving substantial reductions in delay costs (between 59-80\%).
As mentioned in Section~\ref{subsec:benchmark-algorithms}, the deterministic approach does not enforce turn time constraints, which is why we observe operational costs increase and huge delay costs improvements.
This demonstrates how stochastic optimization effectively balances these competing objectives, compared to a method that does not consider delay uncertainty at all.

The diving heuristic consistently outperforms the restricted master approach across all scenario configurations, achieving approximately 62\% total cost reduction versus 45-51\% for the restricted master.

Interestingly, solution quality remains relatively stable as the number of scenarios increases beyond 5, suggesting that even a few well-chosen scenarios can effectively capture delay uncertainty.

These results confirm that explicitly modeling delay uncertainty leads to substantially better outcomes than purely operational cost-focused approaches, with the diving heuristic providing the most effective balance between computational tractability and solution quality.

\begin{result}
    Stochastic optimization significantly outperforms purely operational cost focused approaches. Both heuristics show an effective trade-off between increased operational costs and greatly reduced delay costs, resulting in substantial total cost reductions of up to 62\%. The diving heuristic consistently delivers the best solutions, reducing total costs by approximately 62\% compared to the deterministic baseline.
\end{result}

\section{Conclusion}\label{sec:conclusion}
In this paper, we presented an approach to solve the stochastic tail assignment problem for optimizing operational and delay costs.
We proposed two mathematical programming formulations: a compact mixed-integer program and a set-partitioning formulation.
To solve the latter, we developed a tailored column generation algorithm where the pricing sub-problem is solved using the open-source Julia packages \texttt{ConstrainedShortestPaths.jl} and \texttt{PiecewiseLinearFunctions.jl} that we developed.
This solution framework can be combined with either a restricted master heuristic or a diving heuristic to find good solutions.

Our numerical experiments on real-world Air France instances demonstrated the effectiveness of our method.
The column generation algorithm provided tight lower bounds across all scenarios, while the compact MIP formulation struggled to scale on larger instances.
The diving heuristic consistently produced near-optimal solutions with an average gap of less than 0.5\%, significantly outperforming both the restricted master heuristic and traditional deterministic approaches.

Most significantly, our approach reduced total expected costs by up to 62\% compared to deterministic approaches, demonstrating that explicitly modeling delay uncertainty leads to substantial operational improvements.
This reduction represents a meaningful balance between increased operational costs and greatly decreased delay costs, a trade-off that traditional deterministic approaches cannot achieve.

On larger instances, while our methods still provide better quality solutions than a pure deterministic solver, they start having difficulty scaling, and can take a long time to solve.
Future work could focus on improving the scalability of the proposed methods through algorithmic and data-driven enhancements.
Developing more sophisticated delay propagation models by incorporating historical flight data would also be valuable.
Finally, investigating integration with real-time decision support systems for day-of-operations adjustments represents another promising direction.

From a practical perspective, these results suggest that airlines can significantly improve their resilience to delays while maintaining operational efficiency by adopting stochastic optimization approaches.
The demonstrated cost reductions could translate to improved on-time performance, reduced passenger compensation expenses, and greater schedule reliability in real-world airline operations.

\section*{Acknowledgements}
We would like to express our gratitude to Air France for providing the real-world data that made this research possible. We are particularly thankful to the Operations Research team at Air France for their valuable insights, domain expertise, and feedback throughout this project. Their operational knowledge was instrumental in developing algorithms that address practical challenges in airline operations.

We also thank Merve Bodur and Patrick Jaillet for their helpful comments and suggestions as reviewers of my thesis, which helped improve the quality of this work.

\clearpage
\appendix
\part*{Appendices}
\addcontentsline{toc}{part}{Appendices}
\section{Compact MIP formulation}\label{appendix:compact-mip-formulation}

The compact MIP formulation is commonly used for deterministic aircraft routing problems, with additional constraints to compute and propagate delays along the routing decision variables.
The difficult part in the stochastic case is to formulate the delay cost function in a way that can be given to a MILP solver.
Indeed, the delay cost function being piecewise linear and delay propagation equations having non-linearity, all these need to be linearized.

\subsection{Routing decision variables}
Using the connection graph $\digraph = (\vertexSetBar, \arcSet)$ and its aircraft specific sub-graphs $\disubgraph{\aircraft} = (\vertexSetBar, \arcSubSet{\aircraft})_{\aircraft\in\aircraftSet}$ introduced in the problem setting in subsection~\ref{subsec:event-graph}, we introduce binary decision variables $\decisioni\in\binarySet$ for each aircraft $\aircraft\in\aircraftSet$ and arc $\arc\in\arcSet^\aircraft$.
This variable is equal to $1$ if and only if aircraft $\aircraft$ performs connection $\arc$.

\subsection{Deterministic formulation}
The deterministic version of the tail assignment problem that only minimizes operational costs can be formulated as flow formulation on the connection graph $\digraph$, as detailed in equation~\eqref{eq:deterministic-MILP}.

\begin{subequations}\label{eq:deterministic-MILP}
    \begin{align}
        \min_\solution\, & \sum_{\leg\in\legSet}\sum_{\aircraft\in\aircraftSet}\sum_{\arc\in\outneighbors{\leg}\cap\arcSubSet{\aircraft}}\legCost\decisioni + \sum_{\aircraft\in\aircraftSet}\sum_{\arc\in\arcSubSet{\aircraft}} \connectionCost\decisioni \label{eq:operational-cost}\\
        \text{s.t.}\, & \sum_{\arc\in\inneighbors{\vertex}\cap\arcSubSet{\aircraft}}\decisioni = \sum_{\arc\in\outneighbors{\vertex}\cap\arcSubSet{\aircraft}} \decisioni,\,\forall\vertex\in\activitySchedule,\,\forall\aircraft\in\aircraftSet\label{eq:route-constraint-1}\\
        & \sum_{\arc\in\outneighbors{s}\cap\arcSubSet{\aircraft}}\decisioni = 1,\quad\forall\aircraft\in\aircraftSet\label{eq:route-constraint-2}\\
        & \sum_{\arc\in\inneighbors{t}\cap\arcSubSet{\aircraft}}\decisioni = 1, \quad\forall\aircraft\in\aircraftSet\label{eq:route-constraint-3}\\& \sum_{\aircraft\in\aircraftSet}\sum_{\arc\in\inneighbors{\leg}\cap\arcSubSet{\aircraft}}\decisioni = 1, \quad\forall\leg\in\legSet\label{eq:set-partitionning}\\        
        & \decisioni\in\binarySet,\quad\forall\aircraft\in\aircraftSet,\quad\forall\arc\in\arcSubSet{\aircraft}.
    \end{align}
\end{subequations}

Indeed, most operational constraints described in Section~\ref{sec:tail_assignment_OR:problem_setting} are already implicitly encoded in the connection graph, and do not need to be explicitly formulated.
The only remaining constraints to formulate are routing and set-partitionning ones.
\emph{Route constraints} can be formulated as $\source$-$\sink$ flow constraints on the event graph, as shown in equations~\eqref{eq:route-constraint-1} to~\eqref{eq:route-constraint-3}.
Similarly, we can enforce that all legs are operated exactly once by formulating \emph{set-partitioning} constraints as in equation~\eqref{eq:set-partitionning}.

\begin{remark}
    Note that set-partitioning constraints are only indexed by legs $\leg\in\legSet$, and maintenances do not appear.
    This is because maintenance constraints are already enforced in the connection graph through preprocessing of the sub-graphs $(\disubgraph{\aircraft})_{\aircraft\in\aircraftSet}$.
\end{remark}

\subsection{Delay cost}
To extend the deterministic model to the stochastic tail assignment problem, we need to tackle the non-linear delay costs \eqref{eq:SAA}.
To do so, we introduce additional decision variables $\delayVariable{\activity}$, measuring the arrival delay of each activity $\activity$ under scenario $\scenario\in\scenarioSet$.

For all activity $\activity\in\activitySchedule$ and scenrio $\scenario\in\scenarioSet$, delay propagation equations~\eqref{eq:delay-model/arrival-departure-propagation} and~\eqref{eq:delay-model/departure-arrival-propagation} translate into:
\begin{equation}\label{eq:delay-propagation-nonlinear}
    \delayVariable{\activity} = \delayScenario{\activity} + \sum_{\aircraft\in\aircraftSet}\sum_{\activity^-\in\inneighbors{\activity}\cap\arcSubSet{\aircraft}} \decisioni[\activity^-,\activity][\aircraft](\delayVariable{\activity^-} - \slack{\activity^-}{\activity})^+,
\end{equation}
where $(a)^+ = max(a, 0)$.
The $\max$ operator can be eliminated using two inequalities:
\begin{subequations}\label{eq:delay-constraints}
    \begin{align}
        &\delayVariable{\activity} \geq \delayScenario{\activity} + \sum_{\aircraft\in\aircraftSet}\sum_{\activity^-\in\inneighbors{\activity}\cap\arcSubSet{\aircraft}} \decisioni[(\activity^-,\activity)][\aircraft](\delayVariable{\activity^-} - \slack{\activity^-}{\activity})\label{eq:non-linear-delay-constraint}\\ 
        &\delayVariable{\activity} \geq \delayScenario{\activity}
    \end{align}
\end{subequations}
Constraints~\eqref{eq:non-linear-delay-constraint} being quadratic, the full MIP can either be solved using a solver that can handle quadratic constraints such as Gurobi~\parencite{gurobioptimizationllcGurobiOptimizerReference2024} and SCIP~\parencite{bolusaniSCIPOptimizationSuite2024}, or linearized by hand into a MILP using McCormick inequalities~\parencite{mccormickComputabilityGlobalSolutions1976}.

We recall that the sample average approximation of the delay cost objective can then be written as:
\begin{equation}\label{eq:delay-cost}
    \frac{1}{|S|}\sum_{\scenario\in\scenarioSet}\sum_{\leg\in\legSet} \delayCostFunction{}(\delayVariable{\leg})
\end{equation}
The delay cost function $\delayCostFunction{}$ being non-linear (see subsection~\ref{subsec:delay-costs}) but piecewise linear non-decreasing, it also needs to be linearized.
To do so, we introduce a set of variables for each combination of scenario, activity and slope interval:
\begin{equation}\label{eq:delay-cost-variables}
    \delta_{\activity, \scenario, j}\in\bbR, \quad\forall j\in \intInterval{0}{J},\,\forall \activity\in\activitySchedule,\,\forall\scenario\in\scenarioSet,
\end{equation}
allowing us to express the delay cost as:
\begin{equation}\label{eq:delay-cost-2}
    \frac{1}{|S|}\sum_{\scenario\in\scenarioSet}\sum_{\activity\in\activitySchedule} \sum_{j=1}^J \beta_j \delta_{\activity, \scenario, j},
\end{equation}
subject to the following linearization constraints:
\begin{subequations}\label{eq:delay-cost-linearization}
    \begin{align}
        \delta_{\activity, \scenario, 0} &\leq 0,&\forall\activity\in\activitySchedule,\,\forall\scenario\in\scenarioSet \\
        \delta_{\activity, \scenario, j} &\geq 0,&\forall\activity\in\activitySchedule,\,\forall\scenario\in\scenarioSet,\,\forall j\in [J] \\
        \delta_{\activity, \scenario, j} &\leq \alpha_{j+1} - \alpha_{j},& \forall \activity\in\activitySchedule,\,\forall\scenario\in\scenarioSet,\,\forall j\in [J-1]\\
        \delayVariable{\activity} &= \sum_{j=0}^J \delta_{\activity, \scenario, j},&\forall \activity\in\activitySchedule,\,\forall\scenario\in\scenarioSet.
    \end{align}
\end{subequations}

The complete extended compact MILP (assuming constraints~\eqref{eq:non-linear-delay-constraint} 
are linearized) formulation for the stochastic tail assignment problem is:
\begin{subequations}\label{eq:complete-stochastic-milp}
    \begin{align}
        \min_{\decision,\delta}\quad & \sum_{\leg\in\legSet}\sum_{\aircraft\in\aircraftSet}
        \sum_{\arc\in\outneighbors{\leg}\cap\arcSubSet{\aircraft}}\legCost\decisioni \nonumber\\
        & + \sum_{\aircraft\in\aircraftSet}\sum_{\arc\in\arcSubSet{\aircraft}} \connectionCost\decisioni \nonumber\\
        & + \frac{1}{|S|}\sum_{\scenario\in\scenarioSet}\sum_{\activity\in\activitySchedule} 
        \sum_{j=1}^J \beta_j \delta_{\activity, \scenario, j} \\
        \text{s.t.}\quad & \eqref{eq:route-constraint-1},\eqref{eq:route-constraint-2},
        \eqref{eq:route-constraint-3},\eqref{eq:set-partitionning}, \eqref{eq:delay-constraints}, \eqref{eq:delay-cost-linearization} \\
        & \decisioni\in\binarySet,\quad \forall\aircraft\in\aircraftSet,\,\forall\arc\in\arcSubSet{\aircraft},\\
        & \delta_{\activity, \scenario, j}\in\bbR,\quad \forall j\in \intInterval{0}{J},\,
        \forall \activity\in\activitySchedule,\,\forall\scenario\in\scenarioSet.
    \end{align}
\end{subequations}

The main strength of this formulation is that it is quite straightforward to implement, and it is very effective on small instances, up to $\approx$300 flight legs with 1 scenario.
However, it can struggle to solve larger instances in reasonable time.
Indeed, when the number of aircraft and flight legs increase, the number of arcs in the graph can become quite large, especially for a high number of scenarios, which results in too many variables for solving the formulation in reasonable time.

\subsubsection{Dantzig-Wolfe decomposition}
In formulation \eqref{eq:complete-stochastic-milp}, routing constraints are block constraints decomposable by aircraft, while set-partitioning constraints \eqref{eq:set-partitionning} are complicating constraints linking all aircraft together.
This constraints structure is well-suited for a Dantzig-Wolfe decomposition~\parencite{conejoDecompositionTechniquesMathematical2006}.

Decision variables can be partitioned by aircraft $\solution = (\solution^\aircraft)_{\aircraft\in\aircraftSet}$, and we denote by $\calB_\aircraft$ the block constraints associated with aircraft $\aircraft$:

\begin{align}
    \calB_\aircraft &= \left\{\solution^\aircraft\in\binarySet^{|\arcSubSet{\aircraft}|} \mid \text{\eqref{eq:flow-conservation}, \eqref{eq:source-sink}} \right\}, \\
    \text{with }\forall\vertex\in\activitySchedule,\, &\sum_{\arc\in\inneighbors{\vertex}\cap\arcSubSet{\aircraft}}\decisioni = \sum_{\arc\in\outneighbors{\vertex}\cap\arcSubSet{\aircraft}} \decisioni \label{eq:flow-conservation} \\
    &\sum_{\arc\in\outneighbors{s}\cap\arcSubSet{\aircraft}}\decisioni = \sum_{\arc\in\inneighbors{t}\cap\arcSubSet{\aircraft}}\decisioni = 1 \label{eq:source-sink}
\end{align}

By definition, each element of $\calB_\aircraft$ is a feasible aircraft route $\route\in\feasibleRouteSet{\aircraft}$ for aircraft $\aircraft$, i.e. a feasible path in the associated connection sub-graph $\disubgraph{\aircraft}$.
We denote by $(z_\route^\aircraft)_{\route\in\feasibleRouteSet{\aircraft}}$ these elements.

A Dantzig-Wolfe decomposition then introduces new decision variables $\alpha_\route^\aircraft$, indicating if route $\route\in\feasibleRouteSet{\aircraft}$ is operated by aircraft $\aircraft\in\aircraftSet$.
By definition, we have that:
\begin{equation}
    \forall\aircraft\in\aircraftSet,\,\forall\arc\in\arcSubSet{\aircraft},\,\solution_\arc^\aircraft = \sum_{\route\ni\arc}\alpha_\route^\aircraft.
\end{equation}
The notation $\route\ni\arc$, means we sum over all routes $\route\in\feasibleRouteSet{\aircraft}$ that contain arc $\arc$.

We can replace it in \eqref{eq:complete-stochastic-milp} to obtain the following set-partitioning formulation:
\begin{subequations}
    \begin{align}
        \min_\alpha\quad & \sum_{\aircraft\in\aircraftSet}\sum_{\route\in\feasibleRouteSet{\aircraft}} c_{\route}^\aircraft \alpha_\route^\aircraft\\
        \text{s.t.}\quad & \sum_{\aircraft\in\aircraftSet}\sum_{\route\ni\leg, \route\in\feasibleRouteSet{\aircraft}} \alpha_\route^\aircraft = 1, & \forall\leg\in\legSet\label{eq:coll1}\\
        & \sum_{\route\in\feasibleRouteSet{\aircraft}} \alpha_\route^\aircraft = 1, & \forall\aircraft\in\aircraftSet\label{eq:coll2}\\
        & \alpha_\route^\aircraft\in\binarySet, & \forall\aircraft\in\aircraftSet,\,\forall\route\in\feasibleRouteSet{\aircraft}.
    \end{align}
\end{subequations}

\section{Restricted master heuristic}\label{appendix:restricted-master-heuristic}
In order to find a good solution of~\eqref{eq:column-generation}, we can use a restricted master heuristic.
For this, we retrieve the generated columns from Algorithm~1 and solve~\eqref{eq:column-generation} restricted to these columns.
This heuristic has the advantage of not needing to implement a full branch-and-price to find the optimal solution, while giving a good quality solution in practice.
\begin{equation}\label{eq:restricted-master-heuristic}
    \begin{aligned}
        \min_\decision\quad & \sum_{\aircraft\in\aircraftSet}\sum_{\route\in\feasibleRouteSubSet{\aircraft}} c_{\route}^\aircraft \decision_\route^\aircraft\\
        \text{s.t.}\quad & \sum_{\aircraft\in\aircraftSet}\sum_{\route\ni\leg, \route\in\feasibleRouteSubSet{\aircraft}} \decision_\route^\aircraft = 1, & \forall\leg\in\legSet &\\
        & \sum_{\route\in\feasibleRouteSubSet{\aircraft}} \decision_\route^\aircraft \leq 1, & \forall\aircraft\in\aircraftSet&\\
        & \decision_\route^\aircraft\in\binarySet, & \forall\aircraft\in\aircraftSet,\,\forall\route\in\feasibleRouteSubSet{\aircraft} &
    \end{aligned}
\end{equation}

\clearpage
\printbibliography

@article{barnhartApplicationsOperationsResearch2003,
  title = {Applications of {{Operations Research}} in the {{Air Transport Industry}}},
  author = {Barnhart, Cynthia and Belobaba, Peter and Odoni, Amedeo R.},
  year = {2003},
  journaltitle = {Transportation Science},
  shortjournal = {Transportation Science},
  volume = {37},
  number = {4},
  pages = {368--391},
  issn = {0041-1655, 1526-5447},
  doi = {10.1287/trsc.37.4.368.23276},
  url = {https://pubsonline.informs.org/doi/10.1287/trsc.37.4.368.23276},
  urldate = {2024-08-27},
  langid = {english}
}

@article{biroliniDayaheadAircraftRouting2023,
  title = {Day-Ahead Aircraft Routing with Data-Driven Primary Delay Predictions},
  author = {Birolini, Sebastian and Jacquillat, Alexandre},
  year = {2023},
  journaltitle = {European Journal of Operational Research},
  volume = {310},
  number = {1},
  pages = {379--396},
  publisher = {Elsevier},
  url = {https://www.sciencedirect.com/science/article/pii/S0377221723001698},
  urldate = {2024-04-09}
}

@online{bolusaniSCIPOptimizationSuite2024,
  title = {The {{SCIP Optimization Suite}} 9.0},
  author = {Bolusani, Suresh and Besançon, Mathieu and Bestuzheva, Ksenia and Chmiela, Antonia},
  year = {2024},
  url = {https://arxiv.org/abs/2402.17702v1},
  urldate = {2024-10-08},
  langid = {english},
  organization = {arXiv.org}
}

@book{conejoDecompositionTechniquesMathematical2006,
  title = {Decomposition Techniques in Mathematical Programming: Engineering and Science Applications},
  shorttitle = {Decomposition Techniques in Mathematical Programming},
  author = {Conejo, Antonio J. and Castillo, Enrique and Minguez, Roberto and Garcia-Bertrand, Raquel},
  year = {2006},
  publisher = {Springer Science \& Business Media},
  url = {https://books.google.fr/books?hl=fr&lr=&id=gdJDAAAAQBAJ&oi=fnd&pg=PA3&dq=conejo+decoposition+2006&ots=rSDiVJ5iD5&sig=wVIVvh6pa8pPBZiP5yNJDhWMv7A},
  urldate = {2025-01-19}
}

@book{desaulniersColumnGeneration2006,
  title = {Column Generation},
  author = {Desaulniers, Guy and Desrosiers, Jacques and Solomon, Marius M.},
  year = {2006},
  volume = {5},
  publisher = {Springer Science \& Business Media},
  url = {https://books.google.com/books?hl=fr&lr=&id=8-5fCG2lumcC&oi=fnd&pg=PR5&dq=column+generation+desaulniers&ots=BYsH4-MaKl&sig=5gR8sA7zjfLQFCOdTZZKFCCsE3s},
  urldate = {2024-11-08}
}

@incollection{gabteniHybridColumnGeneration2006,
  title = {A {{Hybrid Column Generation}} and {{Constraint Programming Optimizer}} for the {{Tail Assignment Problem}}},
  booktitle = {Integration of {{AI}} and {{OR Techniques}} in {{Constraint Programming}} for {{Combinatorial Optimization Problems}}},
  author = {Gabteni, Sami and Grönkvist, Mattias},
  editor = {Beck, J. Christopher and Smith, Barbara M.},
  editora = {Hutchison, David and Kanade, Takeo and Kittler, Josef and Kleinberg, Jon M. and Mattern, Friedemann and Mitchell, John C. and Naor, Moni and Nierstrasz, Oscar and Pandu Rangan, C. and Steffen, Bernhard and Sudan, Madhu and Terzopoulos, Demetri and Tygar, Dough and Vardi, Moshe Y. and Weikum, Gerhard},
  editoratype = {redactor},
  year = {2006},
  volume = {3990},
  pages = {89--103},
  publisher = {Springer Berlin Heidelberg},
  location = {Berlin, Heidelberg},
  doi = {10.1007/11757375_9},
  url = {http://link.springer.com/10.1007/11757375_9},
  urldate = {2024-08-27},
  isbn = {978-3-540-34306-6 978-3-540-34307-3}
}

@article{glombStochasticOptimizationApproach2023,
  title = {A Stochastic Optimization Approach for Optimal {{Tail Assignment}} with Knowledge-Based Predictive Maintenance},
  author = {Glomb, Lukas and Liers, Frauke and Rösel, Florian},
  year = {2023},
  journaltitle = {CEAS Aeronautical Journal},
  shortjournal = {CEAS Aeronaut J},
  volume = {14},
  number = {3},
  pages = {715--728},
  issn = {1869-5582, 1869-5590},
  doi = {10.1007/s13272-023-00663-0},
  url = {https://link.springer.com/10.1007/s13272-023-00663-0},
  urldate = {2024-08-27},
  langid = {english}
}

@book{gronkvistTailAssignmentProblem2005a,
  title = {The Tail Assignment Problem},
  author = {Grönkvist, Mattias},
  year = {2005},
  publisher = {Citeseer},
  url = {https://citeseerx.ist.psu.edu/document?repid=rep1&type=pdf&doi=e91781751909e9f68744c9f6558fa8e68e083aa0},
  urldate = {2024-08-27}
}

@software{gurobioptimizationllcGurobiOptimizerReference2024,
  title = {Gurobi {{Optimizer Reference Manual}}},
  author = {Gurobi Optimization, LLC},
  year = {2024},
  url = {https://www.gurobi.com}
}

@article{haneFleetAssignmentProblem1995,
  title = {The Fleet Assignment Problem: {{Solving}} a Large-Scale Integer Program},
  shorttitle = {The Fleet Assignment Problem},
  author = {Hane, Christopher A. and Barnhart, Cynthia and Johnson, Ellis L. and Marsten, Roy E. and Nemhauser, George L. and Sigismondi, Gabriele},
  year = {1995},
  journaltitle = {Mathematical Programming},
  shortjournal = {Mathematical Programming},
  volume = {70},
  number = {1--3},
  pages = {211--232},
  issn = {0025-5610, 1436-4646},
  doi = {10.1007/BF01585938},
  url = {http://link.springer.com/10.1007/BF01585938},
  urldate = {2024-09-06},
  langid = {english}
}

@incollection{irnichShortestPathProblems2005,
  title = {Shortest {{Path Problems}} with {{Resource Constraints}}},
  booktitle = {Column {{Generation}}},
  author = {Irnich, Stefan and Desaulniers, Guy},
  editor = {Desaulniers, Guy and Desrosiers, Jacques and Solomon, Marius M.},
  year = {2005},
  pages = {33--65},
  publisher = {Springer-Verlag},
  location = {New York},
  doi = {10.1007/0-387-25486-2_2},
  url = {http://link.springer.com/10.1007/0-387-25486-2_2},
  urldate = {2023-12-18},
  isbn = {978-0-387-25485-2},
  langid = {english}
}

@article{khaledCompactOptimizationModel2018,
  title = {A Compact Optimization Model for the Tail Assignment Problem},
  author = {Khaled, Oumaima and Minoux, Michel and Mousseau, Vincent and Michel, Stéphane and Ceugniet, Xavier},
  year = {2018},
  journaltitle = {European Journal of Operational Research},
  volume = {264},
  number = {2},
  pages = {548--557},
  publisher = {Elsevier},
  url = {https://www.sciencedirect.com/science/article/pii/S0377221717305866},
  urldate = {2024-08-27}
}

@article{khaledMulticriteriaRepairRecovery2018,
  title = {A Multi-Criteria Repair/Recovery Framework for the Tail Assignment Problem in Airlines},
  author = {Khaled, Oumaima and Minoux, Michel and Mousseau, Vincent and Michel, Stéphane and Ceugniet, Xavier},
  year = {2018},
  journaltitle = {Journal of Air Transport Management},
  volume = {68},
  pages = {137--151},
  publisher = {Elsevier},
  url = {https://www.sciencedirect.com/science/article/pii/S0969699717300698},
  urldate = {2024-08-27}
}

@article{lanPlanningRobustAirline2006a,
  title = {Planning for {{Robust Airline Operations}}: {{Optimizing Aircraft Routings}} and {{Flight Departure Times}} to {{Minimize Passenger Disruptions}}},
  shorttitle = {Planning for {{Robust Airline Operations}}},
  author = {Lan, Shan and Clarke, John-Paul and Barnhart, Cynthia},
  year = {2006},
  journaltitle = {Transportation Science},
  shortjournal = {Transportation Science},
  volume = {40},
  number = {1},
  pages = {15--28},
  issn = {0041-1655, 1526-5447},
  doi = {10.1287/trsc.1050.0134},
  url = {https://pubsonline.informs.org/doi/10.1287/trsc.1050.0134},
  urldate = {2024-10-08},
  langid = {english}
}

@article{marlaRobustOptimizationLessons2018,
  title = {Robust Optimization: {{Lessons}} Learned from Aircraft Routing},
  shorttitle = {Robust Optimization},
  author = {Marla, Lavanya and Vaze, Vikrant and Barnhart, Cynthia},
  year = {2018},
  journaltitle = {Computers \& Operations Research},
  volume = {98},
  pages = {165--184},
  publisher = {Elsevier},
  url = {https://www.sciencedirect.com/science/article/pii/S0305054818301023},
  urldate = {2024-08-27}
}

@article{mccormickComputabilityGlobalSolutions1976,
  title = {Computability of Global Solutions to Factorable Nonconvex Programs: {{Part I}} — {{Convex}} Underestimating Problems},
  shorttitle = {Computability of Global Solutions to Factorable Nonconvex Programs},
  author = {McCormick, Garth P.},
  year = {1976},
  journaltitle = {Mathematical Programming},
  shortjournal = {Mathematical Programming},
  volume = {10},
  number = {1},
  pages = {147--175},
  issn = {0025-5610, 1436-4646},
  doi = {10.1007/BF01580665},
  url = {http://link.springer.com/10.1007/BF01580665},
  urldate = {2024-10-08},
  langid = {english}
}

@article{parmentierAircraftRoutingCrew2020,
  title = {Aircraft Routing and Crew Pairing: {{Updated}} Algorithms at {{Air France}}},
  shorttitle = {Aircraft Routing and Crew Pairing},
  author = {Parmentier, Axel and Meunier, Frédéric},
  year = {2020},
  journaltitle = {Omega},
  shortjournal = {Omega},
  volume = {93},
  pages = {102073},
  issn = {0305-0483},
  doi = {10.1016/j.omega.2019.05.009},
  url = {https://www.sciencedirect.com/science/article/pii/S0305048317306837},
  urldate = {2022-02-01},
  langid = {english}
}

@unpublished{parmentierAlgorithmsNonLinearStochastic2017,
  title = {Algorithms for {{Non-Linear}} and {{Stochastic Resource Constrained Shortest Paths}}},
  author = {Parmentier, Axel},
  year = {2017},
  eprint = {1504.07880},
  eprinttype = {arXiv},
  eprintclass = {cs},
  url = {http://arxiv.org/abs/1504.07880},
  urldate = {2022-04-01}
}

@article{sadykovPrimalHeuristicsBranch2019,
  title = {Primal {{Heuristics}} for {{Branch}} and {{Price}}: {{The Assets}} of {{Diving Methods}}},
  shorttitle = {Primal {{Heuristics}} for {{Branch}} and {{Price}}},
  author = {Sadykov, Ruslan and Vanderbeck, François and Pessoa, Artur and Tahiri, Issam and Uchoa, Eduardo},
  year = {2019},
  journaltitle = {INFORMS Journal on Computing},
  shortjournal = {INFORMS Journal on Computing},
  volume = {31},
  number = {2},
  pages = {251--267},
  issn = {1091-9856, 1526-5528},
  doi = {10.1287/ijoc.2018.0822},
  url = {https://pubsonline.informs.org/doi/10.1287/ijoc.2018.0822},
  urldate = {2025-01-15},
  langid = {english}
}

@report{uchoaOptimizingColumnGeneration2024,
  title = {Optimizing with {{Column Generation}}: {{Advanced Branch-Cut-and-Price Algorithms}} ({{Part I}})},
  author = {Uchoa, Eduardo and Pessoa, Artur and Moreno, Lorenza},
  year = {2024},
  institution = {Cadernos do LOGIS-UFF, Universidade Federal Fluminense, Engenharia de Produ},
  url = {https://optimization-online.org/2024/08/optimizing-with-column-generation-advanced-branch-cut-and-price-algorithms-part-i/}
}

@article{yanRobustAircraftRouting2018,
  title = {Robust {{Aircraft Routing}}},
  author = {Yan, Chiwei and Kung, Jerry},
  year = {2018},
  journaltitle = {Transportation Science},
  shortjournal = {Transportation Science},
  volume = {52},
  number = {1},
  pages = {118--133},
  issn = {0041-1655, 1526-5447},
  doi = {10.1287/trsc.2015.0657},
  url = {https://pubsonline.informs.org/doi/10.1287/trsc.2015.0657},
  urldate = {2024-04-09},
  langid = {english}
}

\end{document}